\documentclass[11pt]{amsart}

\usepackage{amsmath}
\usepackage{amsthm}
\usepackage{stmaryrd}
\usepackage{amsfonts,mathrsfs,amssymb}
\usepackage{times}
\usepackage{fullpage}

\title{Bergman interpolation on finite Riemann surfaces.\\
Part II:  Poincar\'e-Hyperbolic Case}
\author{Dror Varolin} 
\email{dror@math.sunysb.edu}
\address{Department of Mathematics \newline \indent Stony Brook University \newline \indent Stony Brook, NY 11794-3651}

\newcommand{\noi}{\noindent}

\newcommand{\co}{{\mathcal O}}

\newcommand{\sa}{{\mathscr A}}

\newcommand{\sC}{{\mathscr C}}

\newcommand{\se}{{\mathscr E}}

\newcommand{\sh}{{\mathscr H}}

\newcommand{\sL}{{\mathscr L}}

\newcommand{\sr}{{\mathscr R}}

\newcommand{\fl}{{\mathfrak l}}

\newcommand{\fp}{{\mathfrak p}}

\newcommand{\vp}{\varphi} 
\newcommand{\ve}{\varepsilon}

\newcommand{\A}{{\mathbb A}}

\newcommand{\C}{{\mathbb C}}
\newcommand{\D}{{\mathbb D}}

\newcommand{\h}{{\mathbb H}}

\newcommand{\N}{{\mathbb N}}
\newcommand{\p}{{\mathbb P}}

\newcommand{\R}{{\mathbb R}}

\newcommand{\Z}{{\mathbb Z}}

\newcommand{\red}{\hfill $\diamond$}

\newcommand{\di}{\partial}
\newcommand{\dbar}{\bar \partial}

\newcommand{\re}{{\rm Re\ }}
\newcommand{\im}{{\rm Im\ }}

\newcommand{\relcomp}{\subset \subset}

\newcommand{\ii}{\sqrt{-1}}

\newcommand{\tensor}{\otimes}

\def\XXint#1#2#3{{\setbox0=\hbox{$#1{#2#3}{\int}$} 
\vcenter{\hbox{$#2#3$}}\kern-.5\wd0}}

\usepackage{hyperref}

\begin{document}

\theoremstyle{plain}
\newtheorem{thm}{\sc Theorem}
\newtheorem*{s-thm}{\sc Theorem}
\newtheorem{lem}{\sc Lemma}[section]
\newtheorem{d-thm}[lem]{\sc Theorem}
\newtheorem{prop}[lem]{\sc Proposition}
\newtheorem{cor}[lem]{\sc Corollary}

\theoremstyle{definition}
\newtheorem{conj}[lem]{\sc Conjecture}
\newtheorem{prob}[lem]{\sc Open Problem}
\newtheorem{defn}[lem]{\sc Definition}
\newtheorem*{s-defn}{\sc Definition}
\newtheorem{qn}[lem]{\sc Question}
\newtheorem{ex}[lem]{\sc Example}
\newtheorem{rmk}[lem]{\sc Remark}
\newtheorem*{s-rmk}{\sc Remark}
\newtheorem{rmks}[lem]{\sc Remarks}
\newtheorem*{ack}{\sc Acknowledgment}

\maketitle

\setcounter{tocdepth}2

\vskip .3in

\begin{abstract}
We formulate the Bergman-type interpolation problem on finite open Riemann surfaces covered by the unit disk.  Our version of the interpolation problem generalizes Bergman-type interpolation problems previously studied by Seip, Berntsson, Ortega Cerd\`a, and a number of other authors.  We then prove sufficient conditions for a sequence to be interpolating.  When the curvature of the weight in question is bounded in an appropriate sense, we show that the sufficient conditions are almost necessary, but not quite.  The results extend work of Ortega Cerd\`a, who resolved the case in which the boundary of the surface is pure $1$-dimensional.  Our version of the interpolation problem effectively changes the geometry of the underlying space near the $0$-dimensional boundary components, or punctures, thereby linking in a crucial way with the previous article in this two-part series.
\end{abstract}



\section*{Introduction}

In this sequel to our paper \cite{v-rs1}, we continue our investigation of interpolation in Bergman spaces over finite Riemann surfaces.  Recall that for an open Riemann surface $X$ with conformal metric $\omega$ and weight function $\psi$, we defined the so-called (generalized) Bergman space   
\[
\sh ^2 (X, e^{-\psi}\omega) := \left \{ g \in \co (X)\ ;\ \int _X |g|^2 e^{-\psi} \omega < +\infty \right \}.
\]
In \cite{v-rs1} we defined another Hilbert space that measures the size of data along a closed discrete subset $\Gamma \subset X$.  In the present article, we change this Hilbert space slightly.  To define this slightly modified Hilbert space, we first recall that the {\it pointwise injectivity radius} of $x \in X$ is the number 
\[
\iota _{\omega}(x) := \sup \left \{ r>0 \ ;\ D^{\omega}_r (x) \text{ is contractible}\right \},
\]
where $D^{\omega}_r (x)$ denotes the set of all points whose $\omega$-geodesic distance to $x$ is less than $r$.  We define
\[
\hat \iota _{\omega}(x) := \min (\iota _{\omega}(x), 1),
\]
and let 
\[
A_{\omega}(x) := \int _{D_{\hat \iota _{\omega}(x)}(x)} \omega
\]
be the $\omega$-area of the disk $D_{\hat \iota _{\omega}(x)}(x)$.  Finally, given a closed discrete subset $\Gamma \subset X$, we  set 
\[
\ell ^2 (\Gamma , e^{-\psi}) := \left \{ f :\Gamma \to \C \ ;\ \sum _{\gamma \in \Gamma} |f(\gamma)|^2e^{-\psi(\gamma)}A_{\omega}(\gamma) < +\infty \right \}.
\]
\begin{s-rmk}
When $(X,\omega)$ is asymptotically flat, which was the case in \cite{v-rs1}, the areas $A_{\omega}(\gamma)$ are uniformly bounded above and below by a positive constants, and thus, from the point of view of the interpolation problem, our definitions here are essentially  generalizations of those of \cite{v-rs1}.  There are, of course, other possible generalizations, but this natural definition allows us to prove rather strong results.
\red
\end{s-rmk}
\noi We say that $\Gamma$ is an interpolation sequence (for the weight function $\psi$ and conformal metric $\omega$) if the restriction map 
\[
\sr _{\Gamma} : \sh ^2(X,e^{-\psi}\omega) \to \ell ^2 (\Gamma, e^{-\psi})
\]
is surjective.  For a given triple $(X, \omega , \psi)$, the goal is to characterize interpolation sequences $\Gamma$ in terms of geometric properties of $\Gamma$, preferably expressed using the metric $\omega$ and the weight function $\psi$.

In the present article, we focus on finite open Riemann surfaces $X$ whose universal cover is the unit disk (such Riemann surfaces were called {\it Poincar\'e-hyperbolic} in \cite{v-book}), and which have smooth boundary consisting of compact, $1$-dimensional boundary components (which we call {\it border curves}) and $0$-dimensional boundary components (which we call {\it punctures}).  Every such Riemann surface (or more generally, every Riemann surface covered by the disk) possesses a unique metric $\omega _P$ of curvature $-4$, which we call the Poincar\'e metric (hence the name "Poincar\'e-hyperbolic").  

The main result of the paper is the following theorem.

\begin{thm}\label{main}
Let $X$ be a finite open Riemann surface covered by the disk, with its Poincar\'e metric $\omega _P$.  Let $\vp \in \sC ^2(X)$ be a weight function satisfying the following conditions:  there exist positive constants $m$ and $M$ such that 
\begin{enumerate}
\item[($\star$)] each puncture $p_j$ is an isolated boundary point of an open set $P_j\subset X$ that is biholomorphic to the punctured disk, such that 
\[
4 \omega _P(\zeta) + m \omega^j _c(\zeta) \le \Delta \vp (\zeta) \le M\omega^j _c(\zeta), \quad \zeta \in P_j,
\]
where $\omega^j _c$ is the cylindrical metric in $P_j$, and 
\item[(B)] each border curve $\sC _j$ is the outer boundary of an open set $A_j$ that is biholomorphic to an annulus, such that 
\[
m \omega _P (\zeta) \le \Delta \vp (\zeta) -2 \omega _P(\zeta) \le M \omega _P, \quad \zeta \in A_j.
\]
\end{enumerate}
Let $\Gamma \subset X$ be a closed discrete subset.  Then the restriction map $\sr _{\Gamma} :\sh ^2 (X, e^{-\vp}\omega _P) \to \ell ^2 (\Gamma, e^{-\vp})$ is surjective if 
\begin{enumerate}
\item[(i+)] $\Gamma$ is uniformly separated, and 
\item[(ii+)] the asymptotic upper density $D^+_{\vp}(\Gamma)$ of $\Gamma$ is strictly less than $1$.
\end{enumerate}
Conversely, if $\sr _{\Gamma}$ is surjective, then 
\begin{enumerate}
\item[(i-)] $\Gamma$ is uniformly separated, and 
\item[(ii-)] $D^+_{\vp}(\Gamma) \le 1$, and if moreover $X$ has no isolated boundary components then $D^+_{\vp}(\Gamma) < 1$.
\end{enumerate}
\end{thm}

\noi A few remarks regarding the precise meaning of some of the terms in Theorem \ref{main} are in order.

\begin{enumerate}
\item[(a)] The cylindrical metric in $\C^*$ is some constant multiple of the metric $\omega _c (z) = (2|z|^2)^{-1}\ii dz \wedge d\bar z$.  In a Riemann surface $X$ covered by the disk, one has special coordinates near punctures; coordinates  that are adapted to the hyperbolic geometry of $X$ inherited from the cover (cf. Section \ref{ends-section}).  In these coordinates, the cylindrical metric is given by the same formula.

\item[(b)] Uniform separation of a closed discrete subset is measured with respect to the geodesic distance of the cylindrical metric near the punctures, and of the Poincar\'e metric near the border curves.

\item[(c)] 
As in the prequel to this article, the asymptotic upper density $D^+_{\vp}(\Gamma)$ is the least upper bound of certain weighted densities of points of $\Gamma$ in large geodesic disks (for the hyperbolic metric, except near the puncture, in which case the disks are geodesic for the cylindrical metric), the least upper bound being taken with respect to the centers of these disks.  Later in the introduction we will give a slightly imprecise version of the definition, and the precise definition will be given in Section \ref{main-proof}, after the definition of density has been given for the punctured disk.

\item[(d)] Since $\vp$ is smooth, there exists a smooth, positive $(1,1)$-form $\Theta$ on $X$ such that $\Delta \vp \ge - \Theta$.
\end{enumerate}

As crucially observed by Ortega Cerd\`a in \cite{quim-rs}, since $\Gamma$ is closed and discrete, defining (and computing) the density reduces to doing so near the boundary of $X$.  In fact, the density remains unchanged if one discards any finite subset of the sequence in question.  Therefore it essentially suffices to define density for the Poincar\'e disk and the Poincar\'e punctured unit disk.  The case of the Poincar\'e disk has been well-studied, but to the author's knowledge the case of the Poincar\'e punctured disk has not been directly considered in interpolation problems until now.  

If $X$ has no punctures (and therefore at least one border curve), Theorem \ref{main} is due to Ortega Cerd\`a \cite{quim-rs}.  We shall discuss Ortega Cerd\`a's theorem in more detail in Subsection \ref{past-results}.  In our previous work \cite{sv1} with A. Schuster, we allowed punctures, but our results carried cumbersome hypotheses which implied, in particular, that our sequences did not accumulate at  the punctures.  Essentially, punctures were omitted in \cite{quim-rs}, and were restrictive in \cite{sv1}, because the Green's function was used to define density, and Green's functions do not see isolated boundary points, as the latter are irregular for the Dirichlet problem.  To some extent, the present article and its predecessor \cite{v-rs1} emerged from a desire to study interpolation along sequences that could accumulate on the punctures.

\begin{s-rmk}
The special case where $X$ has at least one puncture but no border curves can almost be derived, with some work, from the main result of \cite{v-rs1}.  The proof there is very similar to the proof here, but in the present article, we can weaken somewhat the requirements on our weight functions, because the hyperbolic geometry of the punctures lets us make use of a technique introduced by Donelly and Fefferman \cite{df}, and further developed by Ohsawa in a number of articles ( See also \cite{bo-otdf}).  The technique of Donnelly-Fefferman-Ohsawa will be presented, in its most general form, in Subsection \ref{dfo-section}.
\red
\end{s-rmk}

There are two important special cases of Theorem 1, namely when $X = \D$ is the unit disk, and $X= \D^*$ is the punctured disk.  The case of the unit disk, in which condition $(\star)$ of Theorem \ref{main} is vacuous, was treated by a number of authors.  The first results are found in the work of K. Seip in the unweighted (and also the standardly weighted) Bergman disk \cite{s2}.  Berndtsson and Ortega Cerd\'a \cite{quimbo} were the first to treat the weighted case, for which they proved sufficiency.  (Berndtsson and Ortega Cerd\`a did not give an explicit definition of asymptotic density in their paper, but it is effectively  defined there.)  To the author's knowledge, necessity seems never to have been completely written down in the case of the unit disk with general weights, though in the work \cite{quimseep} of Ortega Cerd\'a and Seip there is an essentially complete sketch of how to do it.  We have therefore decided to provide complete details here, where we mostly follow the ideas of Ortega Cerd\'a and Seip, with only minor modifications that suit our own taste; we consider this part of the work to be essentially known.

To the author's knowledge, the case of the punctured disk has never been considered before the present article.  In fact, the punctured disk case contains nearly all the ideas needed to handle the general case.  Roughly speaking, a closed discrete subset of $\D^*$ can be written as a union of two sequences, the first of which only accumulates at the outer boundary, or border, of $\D^*$, and the second of which accumulates only at the puncture.  This decomposition is not unique, but the notions of uniform separation and of upper density are both independent of the decomposition.  

Near the border, a sequence in $\D^*$ looks very much like a sequence in $\D$, so its upper density can be defined, after a small amount of care, as though the sequence is indeed a sequence in $\D$.  However, near the puncture, the geometry one must consider is determined by our definition of the $\ell ^2$-spaces of functions on the sequence.  The rather natural definition we have chosen provides a geometry near the puncture that is very much like the cylindrical geometry considered in \cite{v-rs1}.  We import the definition of density near the puncture from the cylindrical case, though because cylindrical geometry is essentially flat and our spaces are negatively curved, formulating the definition of density for sequences near the puncture requires more care than was needed near the border.  Finally, the density of the sequence $\Gamma$ is defined as the maximum of the density near the puncture and the density near the border.  As we mentioned earlier, the density of a sequence is unchanged if we throw away finitely many points of the sequence, and this is the reason why the upper density is independent of the decomposition of the sequence into border-supported and puncture-supported subsequences.

Defining the density in the general case is now more clear.  The sequence $\Gamma$ is decomposed into a finite part, and a union of "tails".  Each tail, i.e., subsequence which accumulates near at most one boundary component, has an upper density, and this upper density is like the upper density in the disk if the boundary component is $1$-dimensional, and like the upper density in the cylinder if the boundary component is $0$-dimensional.  The upper density is then the maximum of the finite number of upper densities thus obtained.  

To establish the sufficiency of the conditions of Theorem \ref{main} for interpolation, we actually prove a stronger result, which we call strong sufficiency.  The result we obtain is stronger in the sense that we do not require the weight functions to be smooth, or to have Laplacian that is bounded from above.  We do require the strictly positive lower bound near the boundary, which was not always the case in \cite{v-rs1}.  The reason, vaguely speaking, is that there is at present no sharp $L^2$ extension theorem in the setting of manifolds that admit non-trivial functions with self-bounded gradient, i.e., in which one can apply the technique of Donnelly-Fefferman-Ohsawa, discussed in Subsection \ref{dfo-section}.  We hope to return to the sharp $L^2$ extension problem on another occasion.

The article is organized as follows.  

In Section \ref{background-section} we recall some background and establish notation that will be followed in the rest of the article.  In particular, we discuss metrics of constant negative curvature, and then recall the $L^2$ extension theorem, the Donnelly-Fefferman-Ohsawa Technique, some results on weights with bounded Laplacian in the unit disk, and the Poisson-Jensen Formula.  

In Section \ref{disk-section} we state and prove Theorem \ref{main} in the case of the unit disk.  As previously mentioned, this is our own take on what is essentially work of Berndtsson, Ortega Cerd\`a and Seip.  But perhaps most importantly, we precisely formulate the asymptotic upper density for the unit disk.  

In Section \ref{pdisk-section}, which is the longest section of the article, we state and prove Theorem \ref{main} in the case of the punctured disk.  In this section, we develop the most important parts of the article: the cylindrical geometry of the puncture, the decomposition of sequences into those supported near the border and near the puncture, and all the related technical machinery that is needed to treat the two types of sequences, and to glue together data obtained from these subsequences into data for the entire sequence.

In Section \ref{main-proof}, we begin by recalling some geometry of the ends of a finite Riemann surface with punctures.  We then have all the tools we need to complete the proof of Theorem \ref{main}, but before doing so we discuss the special case proved by Ortega Cerd\`a, where $X$ has no punctures.  We then turn our attention to the proof of Theorem \ref{main}.  First, we define the asymptotic upper density.  Then we establish necessity.  Finally we prove a strong sufficiency theorem as in the cases of the disk and the punctured disk, and show how it implies the weaker form of sufficiency required for the completion of the proof of Theorem \ref{main}.

The article ends with a short section that remarks on the equivalence of our interpolation problem with the Shapiro-Shields interpolation problem.

\begin{ack}
I am grateful to Henri Guenancia, Long Li, Jeff McNeal, Quim Ortega Cerd\`a and Alex Schuster for many stimulating conversations both past and present, and without which this work would not have come to be.  I am also grateful to the anonymous referee for very useful and interesting remarks.
\red
\end{ack}

\section{Background}\label{background-section}

Let $X$ be a Riemann surface.  We write  $d^c = \frac{\ii}{2} (\dbar - \di)$, and denote by
\[
\Delta := dd^c = \ii \di \dbar 
\]
the Laplace operator (so normalized).  Note that our Laplacian sends functions to $(1,1)$-forms (or currents, if the functions are only locally integrable). We denote by $\phi _z$ the disk involution sending $0$ to $z$: 
\[
\phi _z (\zeta) := \frac{z-\zeta}{1-\bar z \zeta}.
\]
The function 
\[
(z,w) \mapsto |\phi _z(w)| = |\phi _w(z)|
\]
is called the pseudohyperbolic distance between $z$ and $w$ in $\D$.  We denote by 
\[
D_r(z) := \{ \zeta \in \D\ ;\ |\phi _z(\zeta)| < r\}
\]
the pseudohyperbolic disk of radius $r$ and center $z$.  

\subsection{Complete metrics of constant negative curvature}

Recall that if $\omega$ is a smooth conformal metric, expressed in local coordinates $z$ as $\omega = e^{-\psi(z)} \frac{\ii}{2} dz \wedge d\bar z$, then the curvature ${\rm R}(\omega)$ of $\omega$ is defined as the (global) $(1,1)$-form  
\[
{\rm R}(\omega) = \di \dbar \psi.
\]
We say the curvature is constant (resp. positive, negative) if the (globally defined) function 
\[
\frac{\ii {\rm R}(\omega)}{\omega}  = 2 e^{\psi} \frac{\di ^2 \psi}{\di z \di \bar z}:X \to \R
\]
(sometimes also called the curvature, or Gaussian curvature) is constant (resp. positive, negative).  It is well-known that every Riemann surface admits a complete conformal metric of constant curvature.  This curvature is positive if and only if $X= \p _1$, $0$ if and only if $X$ is covered by the complex plane, and negative if and only if $X$ is covered by the disk.  Thus no open Riemann surface has a complete metric of constant positive curvature, and an open Riemann surface $X$ has a complete flat metric if and only if $X = \C$ or $X=\C^*$.  Up to homothety, in $\C$ there is a unique flat metric.  In $\C^*$ the complete flat metric is unique up to a constant multiple.  In the hyperbolic case, things are even better.  If $X$ is covered by the unit disk, then $X$ has a unique metric of constant curvature $-4$, as we now recall.  

\subsubsection{Existence and uniqueness of the hyperbolic metric}

On the unit disk, one has the Poincar\'e metric 
\[
\omega _P := \frac{\ii dz \wedge d\bar z}{2(1-|z|^2)^2} = \frac{\ii}{2} \di \dbar \log \frac{1}{1-|z|^2},
\]
which is complete and has constant negative curvature equal to $-4$.  The Poincar\'e metric has the additional feature that ${\rm Aut}(\D) \subset {\rm Isom}(\omega _P)$.  It follows that if $X$ is a Riemann surface with covering map $\pi :\D \to X$, then the group of deck transformations $G_{\pi} \subset {\rm Aut}(\D)$ consists of isometries of $\omega _P$, and thus we can push $\omega _P$ forward to $X$ by $\pi$, obtaining a metric that we continue to denote by $\omega _P$, and also call the Poincar\'e metric.  We note that the metric $\omega _P$ is complete on $X$.  Explicitly, 
\[
\omega _P(d\pi (z) \xi, \overline {d\pi(z) \xi}) = \frac{|dz(\xi)|^2}{2(1-|z|^2)^2}.
\]
\begin{rmk}
Sometimes the Riemann surface $X$ may itself be an open subset of another Riemann surface $Y$ that is covered by the unit disk.  In this case, there may be unnecessary confusion in the notation $\omega _P$.  Thus, when we need to specify the surface as well, we may write $\omega _P^X$ for the Poincar\'e metric of $X$.
\red
\end{rmk}
  
Finally, $\omega _P$ is the only metric of constant curvature $-4$.  Indeed, if $\omega_1$ and $\omega _2$ are two complete metrics of constant negative curvature $-c$ on $X$, we may lift them to the unit disk via $\pi$, and if they are equal on $\D$, then they are equal on $X$.  Thus we might as well assume $X= \D$.  Write 
\[
\omega _i = e^{u_i} \frac{\ii}{2} dz \wedge d\bar z, \quad i=1,2.
\]
The remainder of the proof is due to Ahlfors \cite{ahlfors-schwarz}.  We want to show that $u_1=u_2$, and by symmetry it suffices to show that $u_1 \le u_2$.  To establish the latter, let $f _r : D_r(0) \to  \D ; z\mapsto z/r$ and write $v = f_r ^* u_2 + \log r^2$.  Observe that the metric $e^v \frac{\ii}{2} dz \wedge d\bar z = f_r ^* \omega _2$ has constant negative curvature $-c$ on $D_r(0)$, and it is also complete there.  On the other hand, $\omega _1$, while having curvature $-c$ on $D_r(0)$, is of course not complete there.

Let $E \subset D_r(0)$ be the open set of all points where $u_1 > v$.  Set $h := u_1 - v$.  We have 
\[
\Delta h =  c (e^{u_1} - e^{v}) \frac{\ii}{2} dz \wedge d\bar z.
\]
It follows that $h$ is subharmonic on $E$, and therefore cannot take its maximum in any interior point of $E$ in $D_r(0)$.  It must thus assume its maximum on the boundary of $E$.  But at a boundary point of $E$ that lies in $D_r(0)$, we must have $u_1 = v$ by continuity, so the maximum is not achieved in the interior of $D_r(0)$.  On the other hand, on the circle $\di D_r(0)$, $h = -\infty$ because the metric $f_r ^* \omega _2$ is complete in $D_r(0)$ while the metric $\omega _1$ is not.  It follows that $E$ is empty, and therefore $u_1 \le v$.  Letting $r \to 1$, we see that $u_1 \le u_2$, as desired.

\subsection{Hyperbolic and pseudohyperbolic distance of $\D$}

Recall that the $\omega_P$-geodesic distance between two points $z, w\in \D$ is 
\[
{\rm dist}_P(z,w) = \frac{1}{2} \log \frac{1+|\phi_z(w)|}{1-|\phi _z(w)|},
\]
Indeed, since ${\rm Aut}(\D) \subset {\rm Isom}(\omega _P)$, for an appropriate $\theta \in \R$ we have 
\[
{\rm dist}_P (z,w) := {\rm dist}_P(e^{\ii \theta}\phi _z(z),e^{\ii \theta}\phi _z(w)) ={\rm dist}(0, |\phi _z(w)|).
\]
And since the geodesics emanating from the origin are rays,  
\[
{\rm dist}(0,r) = \int _0 ^r \frac{dt}{1-t^2} = \frac{1}{2} \log \frac{1+r}{1-r}.
\]
It follows that the so-called pseudohyperbolic distance $\rho (z,w) := |\phi _z(w)|$ satisfies 
\[
|\phi _z(w)| = \frac{e^{2{\rm dist}_P(z,w)}-1}{e^{2{\rm dist}_P(z,w)}+1} = \tanh \left ({\rm dist}_P(z,w)\right ).
\]
In particular, the hyperbolic distance is monotonically increasing in the pseudohyperbolic distance, and the ratio of the two distances converges to $1$ as the pair of points comes together.  

\subsection{The $L^2$ extension theorem}

Since the work of Ohsawa and Takegoshi \cite{ot-first}, there have been many statements and proofs (as well as applications) of theorems on $L^2$ extension of holomorphic functions and sections of holomorphic line bundles.  We will make use of the following version, proved by the author in \cite{v-tak}.

\begin{d-thm}\label{ot-basic}
Let $(X,\omega)$ be a Stein K\"ahler manifold, and let $Z \subset X$ be a smooth hypersurface.  Assume there exists a section $T \in H^0(X,L_Z)$ and a metric $e^{-\lambda}$ for the line bundle $L_Z \to X$ associated to the smooth divisor $Z$, such that $e^{-\lambda}|_Z$ is still a singular Hermitian metric, and 
\[
\sup _X |T|^2e^{-\lambda} \le 1.
\]
Let $H \to X$ be a holomorphic line bundle with singular Hermitian metric $e^{-\psi}$ such that $e^{-\psi}|_Z$ is still a singular Hermitian metric.  Assume that 
\[ 
\ii (\di \dbar \psi  +{\rm Ricci}(\omega)) \ge \ii \di \dbar \lambda _Z
\]
and
\[
\ii (\di \dbar \psi +{\rm Ricci}(\omega)) \ge(1+ \delta) \ii \di \dbar \lambda _Z
\]
for some positive constant $\delta \le 1$.  Then for any section $f \in H^0(Z,H)$ satisfying 
\[
\int _Z \frac{|f|^2e^{-\psi}}{|dT|_{\omega}^2e^{-\lambda }}dA_{\omega} <+\infty 
\]
there exists a section $F\in H^0(X,H)$ such that 
\[
F|_Z=f \quad \text{and} \quad \int _X |F|^2e^{-\psi} dV_{\omega} \le \frac{24\pi}{\delta}\int _Z \frac{|f|^2e^{-\psi}}{|dT|_{\omega}^2e^{-\lambda }}dA_{\omega}.
\]
\end{d-thm}

\subsection{The theorem of Donnelly-Fefferman-Ohsawa}\label{dfo-section}

Let $(X,\omega)$ be a K\"ahler manifold, $L \to X$ a holomorphic line bundle with singular Hermitian metric $e^{-\vp}$, and $\Omega \relcomp X$ a pseudoconvex domain with smooth boundary.  Suppose there are positive functions $\tau$ and $A$ on $\overline{\Omega}$ with $\tau$ $\sC ^2$-smooth.  One has the following well-known identity.  (See, for example, \cite{v-tak}.)

\begin{d-thm}[Twisted basic estimate]\label{tbk}
For any $L$-valued $(0,1)$-form $\beta$ in the domain of $\dbar ^* _{\psi}$ on $\overline{\Omega}$ one has the estimate 
\begin{eqnarray*}
&& \int _{\Omega} (\tau+A) |\dbar ^*_{\psi}\beta|^2e^{-\psi} dV_{\omega} + \int _{\Omega} \tau |\dbar ^*_{\psi}\beta|^2e^{-\psi} dV_{\omega} \\
&& \qquad \ge \int _{\Omega} \left < \left \{ \tau(\di \dbar \psi + {\rm Ricci}(\omega))- \di \dbar \tau - \frac{\di \tau \wedge \dbar \tau}{A} \right \}\beta , \beta \right > e^{-\psi} dV_{\omega}.
\end{eqnarray*}
\end{d-thm}
The twisted basic estimate is obtained from the Bochner-Kodaira Identity 
\begin{eqnarray*}
\int _{\Omega} |\dbar ^*_{\vp}\beta|^2e^{-\vp} dV_{\omega} + \int _{\Omega} |\dbar \beta|^2e^{-\vp} dV_{\omega} &=& \int _{\Omega}  \left < \left \{\di \dbar \vp + {\rm Ricci}(\omega)\right \}\beta , \beta \right > e^{-\vp} dV_{\omega}\\
&& + \int _{\Omega} |\bar \nabla \beta |^2e^{-\vp} dV_{\omega} + \int _{\di \Omega} \left < \{ \di \dbar \rho \} \beta , \beta \right > e^{-\vp}\frac{dS}{|\di \rho|^2}
\end{eqnarray*}
by substituting $e^{-\vp} = \tau e^{-\psi}$, using the non-negativity of the last two terms, and applying the Cauchy-Schwarz Inequality.  Of course, the Bochner-Kodaira Identity only makes sense for continuous forms, and it is proved for smooth forms in the domain of $\dbar ^*_{\vp}$.  But since the latter are dense in the graph norm, after we use the pseudoconvexity of $\Omega$, it suffices to prove the twisted basic estimate for smooth forms.

If we take $\tau = e^{-\eta}$ and $A = \frac{\tau}{\nu}$ for some smooth function $\eta$ and positive constant $\nu$, then we have the following estimate:  for all $L$-valued $(0,1)$-forms $\beta$ in the domain of $\dbar ^*_{\psi}$, 
\begin{eqnarray}
\nonumber && \frac{1+\nu}{\nu} \int _{\Omega} e^{-\eta} |\dbar ^* _{\psi} \beta|^2e^{-\psi} dV_{\omega} + \int _{\Omega} e^{-\eta} |\dbar \beta|^2e^{-\psi} dV_{\omega} \\
\label{eta-best}&& \qquad \ge \int _{\Omega} e^{-\eta}\left < \left \{ \di \dbar \psi + {\rm Ricci}(\omega)) + \di \dbar \eta - (1+\nu) \di \eta \wedge \dbar \eta \right \}\beta , \beta \right > e^{-\psi} dV_{\omega}.
\end{eqnarray}
With the estimate \eqref{eta-best}, we can now prove the following theorem, due to Ohsawa, which is an analogue for $\dbar$ of a theorem proved by Donelly-Fefferman for the exterior derivative $d$.

\begin{d-thm}[Donnelly-Fefferman, Ohsawa]\label{odf-thm}
Let $(X,\omega)$ be a Stein K\"ahler manifold, $L\to X$ a holomorphic line bundle with singular Hermitian metric $e^{-\psi}$, $\eta \in W^{1,2}_{\ell oc}(X)$ a function, $\nu$ a positive number, and $\Theta$ a non-negative, almost everywhere positive $(1,1)$-form such that
\begin{equation}\label{sub-curv-hyp}
\ii (\di \dbar \psi + {\rm Ricci}(\omega) + \di \dbar \eta - (1+\nu) \di \eta \wedge \dbar \eta )\ge \Theta.
\end{equation}
Then for any $L$-valued $\dbar$-closed $(0,1)$-form $\alpha$ satisfying 
\[
\int _{X} e^{\eta} |\alpha|^2_{\Theta}e^{-\psi} dV_{\omega} < +\infty
\]
there exists a measurable section $u$ of $L$ such that 
\[
\dbar u = \alpha \quad \text{and} \quad \int _X e^{\eta} |u|^2e^{-\psi} dV_{\omega} \le \frac{\nu +1}{\nu} \int _{X} e^{\eta} |\alpha|^2_{\Theta}e^{-\psi} dV_{\omega}.
\]
\end{d-thm}

\begin{proof}
By standard approximation methods, we can replace $X$ by a smoothly bounded pseudoconvex domain $\Omega \relcomp X$, and assume that $e^{-\psi}$ and $\eta$ are smooth functions.  With these reductions, consider the linear functional 
\[
\sL (\dbar ^*_{\psi} \beta) := \int _{\Omega} \left < \beta ,\alpha \right > e^{-\psi} dV_{\omega}
\]
defined on the subspace
\[
{\rm Image}(\dbar ^*_{\psi}) := \{ \dbar ^*_{\psi} \beta\ ;\ \beta \in {\rm Domain}(\dbar ^*_{\psi}) \}.
\]
Since $\alpha \in {\rm Kernel}(\dbar)$, it is orthogonal to the image of $\dbar ^*_{\psi}$ (the formal adjoint of $\dbar$ acting on $L$-valued $(0,2)$-forms), and thus it suffices to restrict $\sL$ to $\dbar ^*_{\psi} \beta$ for $\beta \in {\rm Kernel}(\dbar)$.  But for such $\beta$, we have 
\begin{eqnarray*}
\sL (\dbar ^*_{\psi}\beta) &\le & \left ( \int _{\Omega} e^{\eta} |\alpha|^2_{\Theta} e^{-\psi} dV_{\omega} \right ) \int _{\Omega} e^{-\eta} \left < \{ \Theta \} \beta , \beta\right > e^{-\psi} dV_{\omega} \\
& \le & \frac{\nu+1}{\nu}  \left ( \int _{\Omega} e^{\eta} |\alpha|^2_{\Theta} e^{-\psi} dV_{\omega} \right ) \int _{\Omega} e^{-\eta}|\dbar ^* _{\psi} \beta|^2 e^{-\psi} dV_{\omega}.
\end{eqnarray*}
Therefore $\sL$ is continuous.  Extending by $0$ in ${\rm Image}(\dbar ^*_{\psi})^{\perp}$ and using the Riesz Representation Theorem, we find a section $U$ of $L$ such that 
\[
\int _{\Omega} e^{- \eta} (\dbar ^*_{\psi}\beta ) \bar U e^{-\psi} dV_{\omega} = \int _{\Omega} \left < \beta ,\alpha \right > e^{-\psi} dV_{\omega} \quad \text{and} \quad \int _{\Omega} e^{-\eta} |U|^2e^{-\psi} dV_{\omega} \le \frac{\nu +1}{\nu} \int _{\Omega} e^{\eta}|\alpha|^2_{\Theta}e^{-\psi} dV_{\omega}.
\]
The first of these says that $\dbar (e^{-\eta}U) = \alpha$.  If we let $u = e^{-\eta}U$ then we have 
\[
\dbar u = \alpha \quad \text{and} \quad \int _{\Omega} e^{\eta} |u|^2e^{-\psi} dV_{\omega} \le \frac{\nu +1}{\nu} \int _{\Omega} e^{\eta}|\alpha|^2_{\Theta}e^{-\psi} dV_{\omega}.
\]
This completes the proof.
\end{proof}

In the case of the Poincar\'e unit disk, we can take $\eta = \log \frac{1}{1-|z|^2}$ to obtain the following result, which is stated (in a slightly different but equivalent form) and proved in \cite{quimbo}, where it is attributed to Ohsawa.  

\begin{d-thm}\label{disk-odf-thm}
Let $\vp : \D \to [-\infty, \infty)$ be upper semi-continuous and satisfy 
\[
\ii \di \dbar \vp \ge (2(1+\nu) + c) \omega _P
\]
for some positive numbers $\nu$ and $c$.  Then for any $(0,1)$-form $\alpha$ on $\D$ satisfying 
\[
\int _{\D} |\alpha|^2_{\omega _P}e^{-\vp} \omega _P < +\infty
\]
there exists a locally integrable function $u$ such that 
\[
\dbar u = \alpha \quad \text{and} \quad \int _{\D} |u|^2e^{-\vp}\omega _P \le \frac{\nu +1}{c \nu} \int _{\D} |\alpha|^2_{\omega _P}e^{-\vp} \omega _P
\]
\end{d-thm}

In the case of the Poincar\'e punctured disk $(\D ^*,\omega _P)$, we obtain the following result.

\begin{d-thm}\label{punctured-disk-odf-thm}
Let $\vp : \D ^* \to [-\infty, \infty)$ be upper semi-continuous and satisfy 
\[
\ii \di \dbar \vp \ge (2(1+\nu) + c) \omega _P
\]
for some positive numbers $\nu$ and $c$.  Then for any $(0,1)$-form $\alpha$ on $\D ^*$ satisfying 
\[
\int _{\D^*} |\alpha|^2_{\omega _P}e^{-\vp} \omega _P < +\infty
\]
there exists a locally integrable function $u$ such that 
\[
\dbar u = \alpha \quad \text{and} \quad \int _{\D^*} |u|^2e^{-\vp}\omega _P \le \frac{\nu +1}{c \nu} \int _{\D^*} |\alpha|^2_{\omega _P}e^{-\vp} \omega _P
\]
\end{d-thm}

\begin{proof}
In Theorem \ref{odf-thm} let $X = \D^*$, $L = \co$, $\eta = - \log \log |z|^{-2}$ and $\psi = \vp + \eta$.  Then 
\[
2 \omega_P =  \ii \di \dbar \eta = \ii \di \eta \wedge \dbar \eta,
\]
 and thus 
 \[
 \di \dbar \psi + {\rm Ricci}(\omega _P) + \di \dbar \eta - (1+\nu) \di \eta \wedge \dbar \eta = \di \dbar \vp - 2(1+\nu) \omega _P \ge c\omega _P.
\]
Letting $\Theta:= c\omega _P$ completes the proof.
\end{proof}

\begin{s-rmk}
Note that H\"ormander's Theorem implies these results if $c \ge 2$, but not otherwise.
\red
\end{s-rmk}

In Section \ref{main-proof} we will extend extend theorems \ref{disk-odf-thm} and \ref{punctured-disk-odf-thm} to general finite open Riemann surfaces covered by the unit disk (cf. Theorem \ref{odf-gen}).

\subsection{Weights with bounded Laplacian in $(\D, \omega _P)$}\label{bdd-lap}

We recall some well-known material that will be used in the proof of Theorem \ref{main}.  

We begin with a result on a solution of Poisson's Equation with locally uniform estimates.  A proof can be found in \cite{sv2}.

\begin{lem}\label{ddbar-est}
For each $0 < r <1$ there exists a constant $C=C_{r}>0$ with the following property.  For any  $(1,1)$-form $\theta \in \sC ^2(\C)$ satisfying 
\[
- M \omega _P \le \theta \le M \omega _P,
\]
there exists $u \in \sC ^2 (D_r(0))$ such that 
\[
\Delta u = \theta \quad \text{and} \quad \sup _{D_r(0)} (|u|+|du|_{\omega _P}) \le C M.
\]
\end{lem}

As a corollary, we have following result, established in \cite{quimbo} (see also \cite{sv2}).

\begin{lem}\label{weight-centering}
Let $\vp \in \sC ^2(\D)$ satisfy 
\[
-M\omega _P \le \Delta \vp \le M \omega _P
\]
for some positive constant $M$.  Then for each $r \in (0,1)$ there is a positive constant $C_r$ such that for any $z \in \D$ there is a holomorphic function $F$ satisfying 
\[
F(z) = 0, \quad \quad |2\re F (\zeta) - \vp(\zeta) + \vp (z)| \le C_r, \quad \text{and} \quad |2 \re dF(\zeta) - d\vp(\zeta)|_{\omega _P} \le C_r 
\]
for all $\zeta \in D_r(z)$.
\end{lem}

Lemma \ref{weight-centering} gives the following  generalizations of Bergman's inequality.  (See \cite{quimseep} for a proof.)

\begin{prop}\label{bergman}
Let $\vp \in \sC ^2(\D)$ satisfy 
\[
-M \omega _P \le  \Delta \vp \le M\omega _P.
\]
Then for each $r \in (0,1)$ there exists $C_r=C_r(M)$ such that for all $f \in \sh ^2(\D, e^{-\vp} \omega _P)$, 
\begin{enumerate}
\item[(a)] 
\[
|f(z)|^2e^{-\vp(z)} \le C_r \int _{D_r(z)} |f|^2e^{-\vp} \omega _P,
\]
and 
\item[(b)]
\[
|d(|f|^2e^{-\vp})|_{\omega _P}(z) \le C_r \int _{D_r(z)} |f|^2e^{-\vp} \omega _P.
\]
\end{enumerate}
\end{prop}

\begin{cor}\label{Bergman-sums-disk}
Let $\vp$ be a weight function as in Proposition \ref{bergman}.  If $\Gamma$ is a finite union of uniformly pseudohyperbolically separated sequences then for each $r \in (0,1)$ there exists a constant $C_r=C_r(M)$ such that for all $f \in \sh ^2(\D, e^{-\vp} \omega _P)$, 
\begin{enumerate}
\item[(a)] 
\[
\sum _{\gamma\in \Gamma} |f(\gamma)|^2e^{-\vp(\gamma)} \le C_r \sum _{\gamma\in \Gamma} \int _{D_r(\gamma)} |f|^2e^{-\vp} \omega _P \le \tilde C_r \int _{\D} |f|^2e^{-\vp} \omega _P,
\]
and 
\item[(b)]
\[
\sum _{\gamma\in \Gamma} |d(|f|^2e^{-\vp})(\gamma)| \le C_r \sum _{\gamma\in \Gamma} \int _{D_r(\gamma)} |f|^2e^{-\vp} \omega _P \le \tilde C_r \int _{\D} |f|^2e^{-\vp} \omega _P.
\]
\end{enumerate}
\end{cor}

\subsection{Poisson-Jensen Formula}  

In the proof of necessity of Theorem \ref{disk-char}, we shall make use of the following weighted analogue of the well-known Poisson-Jensen Formula, which gives weighted counts of the number of zeros of a holomorphic function in a large pseudohyperbolic disk.  To formulate it, we denote the Green's function for the unit disk $\D$ with pole at $z$ by
\[
G_z (\zeta) := \log |\phi _z(\zeta)|,
\]
and the pseudohyperbolic disk of radius $r$ by 
\[
D_r(z) := \{ \zeta \in \D\ ;\ G(z,\zeta) < \log r \}.
\]

\begin{d-thm}[Poisson-Jensen Formula]\label{Jensen-hyp}
Fix a weight function $\psi \in \sC ^2 (\D)$.  Let $f \in \co (\D)$, let $z \in \D$, and let $r \in (0,1)$.  Let $a_1,...,a_N$ be the (possibly not distinct) zeros of $f$ in $D_r(z)$, and assume that $f(z) \neq 0$, and that there are no zeros of $f$ on the boundary of the pseudohyperbolic disk $D_r (z)$. Then 
\[
\frac{1}{\pi} \int _{\di D_r(z)} \log (|f|^2e^{-\psi} ) d^c G_z =  \log (|f(z)|^2e^{-\psi(z)}) + \sum _{j=1} ^N \log \frac{r^2}{|\phi _{a_j}(z)|^2} - \frac{1}{2\pi} \int _{D_r(z)} \log \frac{r^2}{|\phi _{z}|^2} \Delta \psi .
\]
\end{d-thm}

\begin{proof}
Recall that $d^c = \frac{\ii}{2}(\dbar - \di)$, so that $dd^c = \Delta$.  Let 
\[
K_z(\zeta) := G_z(\zeta)- \log r \quad \text{and} \quad H(\zeta) = \log \left ( \frac{|f(\zeta)|^2e^{-\vp(\zeta)}}{\prod _{j=1} ^N \frac{|\phi_{a_i}(\zeta)|^2}{r^2}} \right ).
\]
By Stokes' Theorem we have 
\begin{equation}\label{green-id-here}
\int _{\di D^o_r(z)} H d^c K_z - K_zd^c H = \int _{D^o_r(z)} H\Delta K_z - K_z \Delta H.
\end{equation}
Now, $d^cK_z=d^cG_z$, $K_z|_{\di D^o_r(z)}\equiv 0$ and $\Delta K_z = \pi \delta _z$.  It follows that 
\begin{eqnarray*}
\frac{1}{\pi } \int _{\di D^o_r(z)} \log |f|^2e^{-\vp} d^c G_z &=&  \log |f(z)|^2e^{-\vp(z)} +  \sum _{j=1} ^N   \left (  \log \frac{r^2}{|\phi _{a_i}(z)|^2} + 2 \int _{\di D^o_r(z)} K_{a_j} d^c G_z \right )\\
&&  - \frac{1}{\pi}\int _{D^o_r(z)} \log \frac{r}{|\phi_z(\zeta)|} \Delta \vp(\zeta).
\end{eqnarray*}
But since $K_z|_{\di D^o_r(z)} \equiv 0$, and application of \eqref{green-id-here} with $H= K_{a_j}$ gives
\[
\int _{\di D^o_r(z)} K_{a_j} d^c G_z =  \int _{D^o_r(z)} K_{a_j} \Delta K_z - K_z \Delta G_{a_j} = K_{a_j}(z) - K_z(a_j) = 0,
\]
and thus the result follows.
\end{proof}

\section{Interpolation in $(\D , \omega _P)$}\label{disk-section}

We write 
\[
\A_r := \{ \zeta \in \C\ ;\ 1/2< |\zeta| < r\}.
\]
In this section, we prove the following special case of Theorem \ref{main}.

\begin{d-thm}\label{disk-char}
Let $\vp \in \sC ^2 (\D)$ be a weight function satisfying 
\[
m \omega _P \le \Delta \vp -2 \omega _P\le M \omega _P
\]
for some positive constants $m$ and $M$, and let $\Gamma \subset \D$ be a closed discrete subset.  Then the restriction map $\sr _{\Gamma} : \sh ^2 (\D, e^{-\vp} \omega _P) \to \ell ^2 (\Gamma , e^{-\vp})$ is surjective if and only if 
\begin{enumerate}
\item[(i)] $\Gamma$ is uniformly separated with respect to the geodesic distance of $\omega _P$, and 
\item[(ii)] the upper density 
\[
D^+_{\vp} (\Gamma) := \limsup _{r \to 1} \sup _{z\in \D} \frac{2\pi \int _{\phi _z(\A_r)} \log \frac{r^2}{|\phi _z(\zeta)|^2}\delta _{\gamma}(\zeta)}{\int _{D_r(z)} \log \frac{r^2}{|\phi _z(\zeta)|^2}( \Delta \vp (\zeta) - 2\omega _P(\zeta))} < 1.
\]
\end{enumerate}
\end{d-thm}

\begin{rmk}
Since $D_r(z) = \phi _z(D_r(0))$ and ${\rm Aut}(\D) \subset {\rm Isom}(\omega_P)$, the functions $A_{\omega}$ are constant.
\red
\end{rmk}

It is useful to define the {\it pseudohyperbolic  separation radius} 
\[
R_{\Gamma} := \inf\left  \{ \frac{|\phi _{\gamma_1}(\gamma _2)|}{2}\ ;\ \gamma _1 , \gamma _2 \in \Gamma,\ \gamma _1 \neq \gamma _2\right \}
\]
of $\Gamma$, which is of course positive if and only if $\Gamma$ is uniformly separated in the pseudohyperbolic distance.

\subsection{Weights and density}

We begin with the following proposition.

\begin{prop}\label{mean-wt-hilb}
Let $\vp \in \sC^2(\D)$ be a weight function satisfying 
\[
-M \omega _P \le \Delta \vp \le M \omega _P
\]
for some positive constant $M$, and let 
\[
\vp _r (z) := \frac{1}{a_r} \int _{D_r(z)} \vp (\zeta) \log \frac{r^2}{|\phi _z(\zeta)|^2} \omega _P(\zeta) = \frac{1}{a_r} \int _{D_r(0)} \vp (\phi _z(\zeta)) \log \frac{r^2}{|\zeta|^2} \omega _P(\zeta),
\]
where 
\[
a_r := \int _{D_r(0)} \log \frac{r^2}{|\zeta|^2}\omega _P (\zeta) .
\]
Then 
\[
-M \omega _P \le \Delta \vp _r \le M \omega _P,
\]
and there is a constant $C_r>0$ such that for all $z \in \D$, 
\begin{equation}\label{ptwise-wt-comp}
|\vp (z)- \vp _r(z)| \le C_r.
\end{equation}
In particular, the identity map defines bounded linear isomorphisms
\[
\sh ^2(\D , e^{-\vp}\omega _P ) \asymp \sh ^2(\D , e^{-\vp_r}\omega _P ) \quad \text{and} \quad \ell ^2(\Gamma , e^{-\vp}) \asymp \ell ^2(\Gamma, e^{-\vp_r}).
\]
\end{prop}

\begin{proof}
The bounds on $\Delta \vp _r$ are obvious from the second integral representation of $\vp _r$.  Next, since all the conditions are invariant under action by ${\rm Aut}(\D)$, it suffices to prove the estimates \eqref{ptwise-wt-comp} for $z=0$.  But at the origin, this estimate follows easily from the Euclidean case, which was done in \cite{v-rs1}.
\end{proof}

Let $\Gamma \subset \D$ be a closed discrete subset.  Choose any function $T \in \co (\D)$ such that ${\rm Ord}(T)=\Gamma$, and, with  
\begin{equation}\label{cr-defn}
c_r =  \int _{\A_ r} \log \frac{r^2}{|\zeta|^2} \omega _P (\zeta) , \quad r \in (1/2,1), 
\end{equation}
we set
\[
\lambda^T_r (z)= \frac{1}{c_r} \int _{\A_r} \log |T(\phi _z(\zeta))|^2 \log  \frac{r^2}{|\zeta|^2} \omega _P (\zeta) =  \frac{1}{c_r} \int _{\phi _z (\A_r)} \log |T(\zeta)|^2 \log  \frac{r^2}{|\phi _z(\zeta)|^2}\omega _P (\zeta) .
\]

\begin{prop}\label{lambda-lemma-disk}
Let the notation be as above.
\begin{enumerate}
\item[(a)] The functions $\sigma _r^{\Gamma} :\D \to [0,\infty)$  and $S^{\Gamma}_r : \Gamma \to (0,\infty)$ defined by 
\[
\sigma^{\Gamma} _r (z) = |T(z)|^2 e^{-\lambda _r^T(z)} \quad \text{and} \quad S^{\Gamma}_r(\gamma) := |dT(\gamma)|^2_{\omega _P}e^{-\lambda ^T_r (\gamma)}, 
\]
as well as the $(1,1)$-form 
\[
\Upsilon ^{\Gamma} _r := \Delta \lambda _r ^T,
\]
are independent of the choice of $T$.  Moreover, for each $r \in (0,1)$ and $z \in \D$ (and in the case of $S^{\Gamma}_r (z)$, $z \in \Gamma$) the three quantities $\sigma _r^{\Gamma}(z)$, $S_r ^{\Gamma}(z)$ and $\Upsilon _r ^{\Gamma}(z)$ depend only of the finite set $D_r(z) \cap \Gamma$ in the sense that we may use any function $T \in \co (D_r(z))$ satisfying ${\rm Ord}(T) = \Gamma \cap D_r(z)$ to determine these three quantities.
\item[(b)] $\sigma^{\Gamma} _r \le 1$.
\item[(c)] For any $\gamma \in \Gamma$ and any $z \in D_{R^o _{\Gamma}}(\gamma)$ such that $|\phi_{\gamma}(z)| > \ve$, we have the estimate 
\[
\sigma ^{\Gamma} _r (z) \ge C_r \ve ^2.
\]
On the other hand, $\frac{1}{\sigma ^{\Gamma}_r}$ is not locally integrable in any neighborhood of any point of $\Gamma$.
\item[(d)] One has the formula
\begin{equation}\label{density-formula} 
\frac{\Upsilon ^{\Gamma} _r (z)}{2\omega _P(z)} = \frac{2\pi}{c_r} \sum _{\gamma \in \phi _z(\Gamma) \cap \A_r} \log \frac{1}{|\phi _z(\gamma)|^2}.
\end{equation}
\end{enumerate}
\end{prop}

\noi The proof is directly analogous to the corresponding proposition in \cite{v-rs1}, and is left to the reader.

\subsection{Sufficiency}

In this section we prove the following result.

\begin{d-thm}\label{disk-suff}
Let $\vp \in \sC ^2 (\D)$ be a weight function satisfying 
\[
m \omega _P \le \Delta \vp -2 \omega _P\le M \omega _P
\]
for some positive constants $m$ and $M$.  Let $\Gamma \subset \D$ be uniformly separated in the pseudohyperbolic distance, and assume $D^+_{\vp} (\Gamma)  < 1$.  Then the restriction map $\sr _{\Gamma} : \sh ^2 (\D, e^{-\vp} \omega _P) \to \ell ^2 (\Gamma , e^{-\vp})$ is surjective.
\end{d-thm}

In fact, we shall prove a slightly stronger result.

\begin{d-thm}[Strong sufficiency]\label{disk-strong-suff}
Let $\vp \in L ^1_{\ell oc} (\D)$ be a subharmonic weight function satisfying 
\[
\Delta \vp  -2 \omega _P \ge m \omega _P
\]
for some positive constant $m$.  Let $\Gamma \subset \D$ be uniformly separated in the pseudohyperbolic distance, and assume 
\[
\Delta \vp - 2\omega _P \ge (1+\ve) \Upsilon ^{\Gamma}_r
\]
for some positive number $\ve$.  Then the restriction map $\sr _{\Gamma} : \sh ^2 (\D, e^{-\vp} \omega _P) \to \ell ^2 (\Gamma , e^{-\vp})$ is surjective.
\end{d-thm}

In view of Proposition \ref{mean-wt-hilb}, Theorem \ref{disk-suff} follows from Theorem \ref{disk-strong-suff}.

\subsubsection{Local extensions}

We shall need the following lemma.

\begin{lem}\label{loc-extn-disk}
Let $r \in (0,1/2)$ and $z \in \D$, and let $\vp$ be a subharmonic function in the unit disk satisfying 
\[
\Delta \vp \ge 2 \omega _P.
\]
Then there exists a holomorphic function $g_z \in \co (D_r(z))$ such that 
\[ 
g_z(z) = e^{\vp(z)/2} \quad \text{and} \int _{D_r(z)} |g_z|^2e^{-\vp} \omega _P \le 32 \pi r^2.
\]
\end{lem}

\begin{proof}
Note that if $\vp(z) = -\infty$ then we can take $g_z\equiv 0$, so we assume from here on out that $\vp(z) \neq -\infty$.  Let $\psi(\zeta) = \vp(\zeta) +\log \frac{1}{1-|\phi _z(\zeta)|^2}$.  Note that on $D_r(z)$, 
\[
\psi (\zeta) \le \vp (\zeta) + \log \frac{1}{1-r^2}.
\]
We apply Theorem \ref{ot-basic} with $X=D_r(z)$, $T(\zeta) = \frac{\phi _z(\zeta)}{r}$,  $\lambda \equiv 0$ and $\delta = 1$. Since 
\[
|dT(z)|^2_{\omega_P}e^{-\lambda (z)} = r^{-2} \quad \text{and} \quad \ii \di \dbar \psi +{\rm Ricci}(\omega _P) = \ii \di \dbar \vp - 2\omega _P \ge 0 = 2\Delta \lambda,
\]
there exists a function $g_z \in \co (D_r(z))$ such that 
\[
g_z(z) = e^{\vp(z)/2} \quad \text{and} \quad \int _{D_r(z)} |g_z|^2e^{-\psi} \omega _P \le 24\pi r^2.
\]
But then 
\[
\frac{3}{4}\int _{D_r(z)} |g_z|^2 e^{-\vp} \omega _P  \le  (1-r^2) \int _{D_r(z)} |g_z|^2 e^{-\vp} \omega _P \le  \int _{D_r(z)} |g_z|^2 e^{-\psi} \omega _P  \le 24\pi r^2.
\]
The proof is thus complete.
\end{proof}

\subsubsection{Global extension: The proof of Theorem \ref{disk-strong-suff}}

Let $\delta < R_{\Gamma}$ be a positive constant.  Suppose we are given data $f \in \ell ^2(\Gamma, e^{-\vp})$.  For each $\gamma \in \Gamma$, let $g_{\gamma} \in \co (D_{\delta} (\gamma))$ be a function such that 
\[
g_{\gamma}(\gamma) = e^{\vp(\gamma)/2} \quad \text{and} \quad \int _{D_{\delta}(\gamma)} |g_{\gamma}|^2e^{-\vp} \omega _P \le 32 \pi.
\]
Such $g_{\gamma}$ exist by Lemma \ref{loc-extn-disk}.   Now let $\chi \in \sC ^{\infty}(\R)$ be a decreasing function satisfying
\[
\chi(x) = 1 \text{ for } x \le \tfrac{1}{2}, \ \chi (x) = 0 \text{ for }x\ge 1, \text{ and } |\chi '(x)| \le 3 \text{ for all }x.
\]
Consider the function 
\[
\tilde F(\zeta) := \sum _{\gamma \in \Gamma} f(\gamma)e^{-\vp(\gamma)/2} g_{\gamma}(\zeta) \chi \left ( \delta ^{-2} |\phi _{\gamma}(\zeta)|^2\right ).
\]
Then 
\[
\tilde F \in \sC ^{\infty} (\D) \quad \text{and} \quad \tilde F|_{\Gamma} = f.
\]
By Lemma \ref{loc-extn-disk} and the definition of $R_{\Gamma}$, 
\[
\int _{\D} |\tilde F|^2e^{-\vp} \omega _P \le \sum _{\gamma \in \Gamma} |f(\gamma)|^2 e^{-\vp(\gamma)}\int _{D_{\delta}(\gamma)} |g_{\gamma}|^2e^{-\vp} \omega _P \le 32\pi ||f||^2,
\]
so that $\tilde F \in L^2(\D, e^{-\vp} \omega _P)$.  We now correct $\tilde F$ to be holomorphic and still interpolate $f$.  We compute that the (automatically $\dbar$-closed) $(0,1)$-form 
\[
\alpha(\zeta) :=  \dbar \tilde F(\zeta) = \sum _{\gamma \in \Gamma} f(\gamma)e^{-\vp(\gamma)/2} g_{\gamma}(\zeta) \chi '\left ( \delta^{-2} |\phi _{\gamma}(\zeta)|^2\right )\frac{\phi _z (\zeta) \overline{\phi _z'(\zeta)}}{\delta^2}d\bar \zeta.
\]
We therefore seek a solution $u$ of the equation $\dbar u = \alpha$ that lies in $L^2(\D, e^{-\vp}\omega _P)$ and vanishes along $\Gamma$.  To this end, consider the weight function 
\[
\psi  = \vp + \log \sigma _r ^{\Gamma},
\]
where $\sigma_r^{\Gamma}$ is as in Proposition \ref{lambda-lemma-disk}.  Note that $\chi '(x) = 0$ for $|x|\le 1/2$ and that 
\[
|d \phi _z|^2_{\omega _P} = (1-|\phi _z|^2)^2\le 1.
\]
Then by the definition of $R_{\Gamma}$ and Property (c) of Proposition \ref{lambda-lemma-disk},
\begin{eqnarray*}
\int _{\D} |\alpha|^2_{\omega _P} e^{-\psi} \omega _P &=& \sum _{\gamma \in \Gamma} |f(\gamma)|^2e^{-\vp(\gamma)}\int _{D_{\delta}(\gamma)} |g_{\gamma}(\zeta)|^2 e^{-\psi(\zeta)} \left | \chi '\left ( \delta^{-2} |\phi _{\gamma}(\zeta)|^2\right )\right |^2 \frac{|\phi _z (\zeta)|^2 |d\phi _z(\zeta)|^2_{\omega_P}}{\delta^4}\omega _P\\
&\le & 9C_r \delta ^{-4} \sum _{\gamma \in \Gamma} |f(\gamma)|^2e^{-\vp(\gamma)}\int _{D_{\delta}(\gamma)} |g_{\gamma}(\zeta)|^2 e^{-\vp(\zeta)} \omega _P < +\infty.
\end{eqnarray*}
Since 
\[
\Delta \psi - 2 \omega _P \ge \frac{1}{1+\ve}\left ( \Delta \vp - 2\omega _P - (1+\ve)\Upsilon ^{\Gamma} _r\right ) + \frac{\ve}{1+\ve} \left ( \Delta \vp -2\omega _P \right ) \ge \frac{m\ve}{1+\ve} \omega _P,
\]
Theorem \ref{disk-odf-thm} gives us a function $u$ such that 
\[
\dbar u = \alpha \quad \text{and} \quad \int _{\D} |u|^2e^{-\psi} \omega _P < +\infty.
\]
By the smoothness of $\alpha$ and the interior ellipticity of $\dbar$, $u$ is smooth.  By Property (c) of Proposition \ref{lambda-lemma-disk}, in particular the non-integrability of $e^{-\psi}$ along $\Gamma$, $u|_{\Gamma} \equiv 0$.  Furthermore, by (b) of Proposition \ref{lambda-lemma-disk}, 
\[
\int_{\D} |u|^2 e^{-\vp} \omega _P \le \int _{\D}|u|^2e^{-\psi} \omega _P < +\infty.
\]
Finally, set $F = \tilde F- u$.  Then $F \in \co (\D)$, $F|_{\Gamma} = f$, and
\[
\int _{\D}|F|^2e^{-\vp} \omega _P \le 2 \left ( \int _{\D}|\tilde F|^2e^{-\vp} \omega _P + \int _{\D}|u|^2e^{-\vp} \omega _P \right ) < +\infty.
\]
This completes the proof of Theorem \ref{disk-strong-suff}.
\qed

\subsection{Necessity}

In this section we complete the proof of Theorem \ref{disk-char} by proving the following theorem.

\begin{d-thm}\label{necess-disk}
Let $\vp \in \sC^2(\D)$ be a weight function satisfying 
\[
m \omega _P \le \Delta \vp - 2 \omega _P \le M \omega _P
\]
for some positive constants $m$ and $M$, and let $\Gamma \subset \D$ be a closed discrete set.  If 
\[
\sr _{\Gamma} :\sh ^2 (\D, e^{-\vp}\omega _P) \to \ell ^2 (\Gamma, e^{-\vp})
\]
is surjective, then $\Gamma$ is uniformly separated in the hyperbolic distance, and $D^+_{\vp}(\Gamma) < 1$.
\end{d-thm}

\begin{rmk}
As mentioned in the introduction, our proof of Theorem \ref{necess-disk} is an adaptation to the unit disk of the work \cite{quimseep} of Ortega Cerd\`-Seip in the case of the Euclidean plane.
\red
\end{rmk}

\subsubsection{Interpolation constant}\label{disk-interp-const-par}

As we explained in \cite{v-rs1} for the analogous case of the complex plane, if $\sr _{\Gamma} : \sh ^2(\D , e^{-\vp}\omega _P) \to \ell ^2 (\Gamma, e^{-\vp})$ is surjective, then by the Closed Graph Theorem the so-called {\it minimal extension operator} $\se _{\Gamma} : \ell ^2 (\Gamma, e^{-\vp}) \to {\rm Kernel}( \sr _{\Gamma})^{\perp} \subset \sh ^2(\D , e^{-\vp}\omega _P)$ is continuous, and moreover has minimal norm among all extension operators.  The norm 
\[
\sa _{\Gamma} := ||\se _{\Gamma}||
\]
of this minimal extension operator is called the {\it interpolation constant of $\Gamma$}.
 
\subsubsection{Necessity of uniform separation}\label{disk-unif-sep-nec}

Suppose $\Gamma \subset \D$ is an interpolation sequence, and let $\gamma_1, \gamma _2 \in \Gamma$ be any two distinct points.  The function $f : \Gamma \to \C$ defined by  
\[
f(\gamma_1) = e^{-\vp (\gamma _1)/2} \quad \text{and} \quad f(\gamma) = 0, \ \gamma \in \Gamma - \{\gamma _1\},
\] 
lies in (the unit sphere of) $\ell ^2 (\Gamma , e^{-\vp})$, and thus $F := \se _{\Gamma} (f) \in \sh ^2 (\Gamma , e^{-\vp}\omega _P)$ satisfies 
\[
|F(\gamma_1)|^2e^{-\vp(\gamma _1)} = 1, \quad |F(\gamma _2)|^2e^{-\vp(\gamma _2)} = 0, \quad \text{and} \quad \int _{\D} |F|^2 e^{-\vp} \omega _P \le \sa _{\Gamma} ^2.
\]
Consider the functions
\[
g := F \circ \phi _{\gamma_1} \in \co (\D) \quad \text {and} \quad \psi := \vp \circ \phi _{\gamma _1} \in \sC ^2 (\D).
\]
Since $\phi _{\gamma _1}$ is an isometry of $\omega _P$, the weight $\psi$ and $\vp$ satisfy the same curvature bounds.  Moreover, 
\[
|g(0)|^2e^{\psi(0)} = 1, \quad |g(\phi _{\gamma _1}(\gamma _2))|^2e^{-\psi (\phi _{\gamma _1}(\gamma _2))} = 0 \quad \text{and} \quad \int_{\D} |g|^2e^{-\psi} \omega _P = \int _{\D} |F|^2 e^{-\vp} \omega _P \le \sa ^2_{\Gamma}.
\]
By Proposition \ref{bergman}(b), there is a universal constant $C>0$ such that
\[
|\phi _{\gamma_1}(\gamma_2)|^{-1} = \left | \frac{|g(0)|^2e^{\psi(0)}- |g(\phi _{\gamma _1}(\gamma _2))|^2e^{-\psi (\phi _{\gamma _1}(\gamma _2))}}{ 0 - \phi _{\gamma_1}(\gamma_2)} \right | \le C \sa _{\Gamma}^2.
\]
Thus $\Gamma$ is uniformly separated.

\subsubsection{Uniform interpolation at a point}

\begin{lem}\label{1-pt-interp-disk}
Let $\vp: \D \to [-\infty, \infty)$ be an upper semicontinuous weight function satisfying 
\[
\Delta \vp - 2 \omega _P \ge c \omega _P
\]
for some positive constant $c$.  Then there exists a constant $C >0$, depending only on $c$ and not on $\vp$ such that for all $z \in \D$ there exists $F \in \sh ^2(\D, e^{-\vp} \omega _P)$ satisfying 
\[
|F(z)|^2e^{-\vp(z)} = 1 \quad \text{and} \quad \int _{\D} |F|^2e^{-\vp} \omega _P \le C.
\]
\end{lem}

\begin{proof}
Since the density of the one point sequence $\Gamma := \{z\}$ is zero, the result follows from Theorem \ref{disk-strong-suff}.  (The uniformity of the constant, though not explicitly stated in Theorem \ref{disk-strong-suff},  follows from its proof.)
\end{proof}

\subsubsection{Perturbation of interpolation sequences}

As in \cite{v-rs1}, in order to estimate the density of the sequence we shall perturb our interpolation sequence in two ways:  a small perturbation of the points of $\Gamma$, and the addition of a point to $\Gamma$.  In the hyperbolic disk, the estimates are slightly better than their Euclidean kin, owing to the existence of a bounded entire function that realizes a zero of multiplicity $1$ at a given point.

\begin{prop}[Small perturbation]\label{jiggle-disk}
Let $\vp \in \sC ^2 (\D)$ satisfy 
\[
-M \omega _P \le \Delta \vp \le M \omega _P
\]
for some positive constant $M$.  Let $\Gamma = \{ \gamma _1,\gamma _2,... \} \subset \D$ be an interpolation sequence with separation radius $R_{\Gamma}$.  Suppose $\Gamma ' = \{ \gamma _1',\gamma _2', ... \} \subset \D$ is another sequence, and there is a $\delta \in (0,\min(\sa _{\Gamma}^{-1}, R_{\Gamma}))$ such that 
\[
\sup _{i \in \N} |\phi _{\gamma _i}(\gamma _i')| \le \delta ^2.
\]
Then $\Gamma '$ is also an interpolation sequence, and its interpolation constant is at most 
\[
\frac{C\sa _{\Gamma}}{1-\delta \sa _{\Gamma}},
\]
where $C$ is independent of $\Gamma$ and $\vp$ (but depends on $M$).
\end{prop}

\begin{proof}
Using the method of proof of the uniform separation of an interpolation sequence, together with Corollary \ref{Bergman-sums-disk}(b), if $F \in \sh ^2(\D, e^{-\vp}\omega_P)$ we find that 
\begin{equation}\label{comparison-eucl}
\sum _{j=1} ^{\infty} \left | |F(\gamma_j)|^2e^{-\vp(\gamma_j)} - |F(\gamma'_j)|^2e^{-\vp(\gamma'_j)}\right | \lesssim \delta ^2 \int _{\D} |F|^2e^{-\vp} \omega _P.
\end{equation}
The rest of the proof proceeds in the same way as that of Proposition 2.11 of \cite{v-rs1}, which is itself a minor adaptation of Lemma 6 in \cite{quimseep}.
\end{proof}

\begin{prop}[Adding a point]\label{add-disk}
Assume $m\omega _P \le \Delta \vp - 2\omega _P \le M \omega _P$ for some positive constants $m$ and $M$.  Let $\Gamma$ be an interpolation sequence, and let $z \in \D - \Gamma$ satisfy $\inf _{\gamma \in \Gamma} |\phi _z(\gamma)| > \delta$.  Then the sequence $\Gamma_z  := \Gamma \cup \{z\}$ is also an interpolation sequence for $\sh ^2(\D, e^{-\vp}\omega _P)$, and its interpolation constant is bounded above by some constant $K$ which depends only on $m$, $\Gamma$ and $\delta$, and in particular, not on $z$.
\end{prop}

\begin{proof}
It suffices to show that there exists $F \in \sh ^2(\D, e^{-\vp}\omega _P)$ satisfying 
\[
F(z) = e^{\vp(z)/2} \quad \text{and} \quad F|_{\Gamma} \equiv 0
\]
with appropriate norm bounds.  To this end, write 
\[
\psi _z := \vp - \frac{m}{2} \left ( \log\frac{1}{1- |\phi _z|^2}\right ).
\]
Since $\Delta \psi _z -2\omega _P \ge \frac{m}{2} \omega _P$, Lemma \ref{1-pt-interp-disk} provides us with a function $G \in \sh ^2(\D, e^{-\psi_z}\omega _P)$ such that 
\[
G(z) = e^{\vp (z)/2} \quad \text{and} \quad \int _{\D} |G|^2 e^{-\psi_z} \omega _P \le C,
\]
where $C$ does not depend on $z$ or $\Gamma$ (and in fact depends only on $m$).  

Now, since $\psi _z \le \vp$, by Corollary \ref{Bergman-sums-disk}(a) we have the estimate 
\[
\sum _{\gamma \in \Gamma} \frac{|G(\gamma)|^2e^{-\vp(\gamma)}}{|\phi _z(\gamma)|^2} \lesssim \frac{1}{\delta ^2}\sum _{\gamma \in \Gamma} \int _{D_{R_{\Gamma}}(\gamma)} |G|^2e^{-\psi_z} \omega _c \lesssim \frac{1}{\delta ^2}.
\]
Since $\Gamma$ is an interpolation sequence for $\sh ^2 (\D , e^{-\vp}\omega _P)$, there exists $H \in \sh ^2 (\D , e^{-\vp}\omega _P)$ such that 
\[
H(\gamma ) = \frac{G(\gamma)}{\phi _z(\gamma)}, \ \gamma \in \Gamma, \quad \text{and} \quad \int _{\D} |H|^2 e^{-\vp} \omega _P \lesssim \frac{\sa _{\Gamma}^2}{\delta ^2}.
\]
Let $F \in \co (\D)$ be defined by 
\[
F (\zeta) := G(\zeta) - \phi _z(\zeta) H(\zeta).
\]
Then 
\[
|F(z)|^2e^{-\vp(z)} = |G(z)|^2e^{-\vp(z)}=1, \quad \text{and} \quad F(\gamma) = G(\gamma) - \phi _z(\gamma)H(\gamma) = 0
\]
for all $\gamma \in \Gamma$.  Finally, 
\begin{eqnarray*}
\left ( \int _{\D} |F|^2 e^{-\vp} \omega _P \right )^{1/2} &\le& \left ( \int _{\D} |G|^2 e^{-\vp} \omega _P \right )^{1/2} + \left ( \int _{\D} |H(\zeta)|^2|\phi _z(\zeta)|^2 e^{-\vp(\zeta)} \omega _c(\zeta) \right )^{1/2}\\
& \le &  \left ( \int _{\C} |G|^2 e^{-\psi _z} \omega _c \right )^{1/2} + \left ( \int _{\C} |H(\zeta)|^2e^{-\vp(\zeta)}\omega _o(\zeta) \right )^{1/2}\le  \frac{C(1 + \sa _{\Gamma})}{\delta},
\end{eqnarray*}
as desired.
\end{proof}

\subsubsection{Estimate for the density of an interpolation sequence}

We wish to estimate the density of the interpolation sequence $\Gamma$ at an arbitrary point $z \in \D$.  Suppose first that $\inf _{\gamma \in \Gamma}|\phi _{\gamma}(z)| < \min (\sa _{\Gamma}^{-1}, R_{\Gamma})$.  By Proposition \ref{jiggle-disk}, we may replace the nearest point $\gamma_o$ of $\Gamma$ to $z$ with the point $z$, and obtain a new interpolation sequence $\Gamma^1 _z := (\Gamma - \{\gamma _o\} )\cup \{z\}$.  On the other hand, if $\inf _{\gamma \in \Gamma}|\phi _{\gamma}(z)| \ge \min (\sa _{\Gamma}^{-1}, R_{\Gamma})$, then by Proposition \ref{add-disk} the sequence $\Gamma^2_z = \Gamma \cup \{z\}$ is also an interpolating sequence.  In both cases, the interpolation constant remains under control, i.e., it is independent of $z$.  Let us write $\Gamma _z$ for either of the interpolation sequences $\Gamma ^1_z$ or $\Gamma ^2_z$ that arise.

Since $\Gamma _z$ is an interpolation sequence, there is a function $F \in \sh ^2(\D, e^{-\vp}\omega _P)$ such that 
\[
F(z) = e^{\vp(z)/2}, \quad F|_{\Gamma_z - \{z\}}\equiv 0, \quad \text{and} \quad ||F||\lesssim \sa _{\Gamma}.
\]
By the Poisson-Jensen Formula \ref{Jensen-hyp} applied to $f=F$ and $\psi = \vp - \log \frac{1}{1-|\phi _z|^2}$, we have 
\[
\int _{\phi _z (\A_r)} \log \frac{r^2}{|\phi _z(\zeta)|^2} \delta _{\Gamma} \le \int _{D_r(z)} \log \frac{r^2}{|\phi _z(\zeta)|^2}(\Delta \vp (\zeta) - 2\omega _P(\zeta)) + \frac{1}{\pi} \int _{\di D_r(z)} \log (|F|^2e^{-\psi} )d^cG_z.
\]
An application of Proposition \ref{bergman}(a), with disks of pseudohyperbolic radius $1/2$ centered at any point of $\di D_r(z)$, yields the estimate 
\[
\int _{\phi _z (\A_r)} \log \frac{r^2}{|\phi _z(\zeta)|^2} \delta _{\Gamma} \le \int _{D_r(z)} \log \frac{r^2}{|\phi _z(\zeta)|^2}(\Delta \vp (\zeta)-2\omega _P(\zeta)) + \log \frac{1}{(1-r^2)} + C,
\]
and thus, since 
\begin{equation}\label{jensen-est-disk}
\frac{1}{c_r} \int _{\phi _z (\A_r)} \log \frac{r^2}{|\phi _z(\zeta)|^2} \delta _{\Gamma} \le \frac{1}{c_r} \int _{D_r(z)} \log \frac{r^2}{|\phi _z(\zeta)|^2}(\Delta \vp (\zeta)-2\omega _P(\zeta)) + C
\end{equation}
for some constant $C$ that is independent of $r$ and $z$.

The estimate \eqref{jensen-est-disk} shows us that the density of $\Gamma$ is at most $1$.  But we can slightly perturb $\Gamma$ to increase its density at $z$, with the perturbation still an interpolation set. Indeed, by Proposition \ref{jiggle-disk}, we may move all the points of $\Gamma$ a pseudo-hyperbolic distance at most a sufficiently small number $\delta$ towards $z$, and the resulting sequence will still be interpolating, with interpolation constant controlled by that of the original sequence $\Gamma$.  To this end, choose $\delta > 0$ sufficiently small according to Proposition \ref{jiggle-disk}, and consider the sequence 
\[
\Gamma^{\delta} _z := \left \{ \phi _z \left ( \frac{\delta - |\phi _z(\gamma)|}{1-\delta |\phi _z(\gamma)|}\frac{\phi _z(\gamma)}{|\phi _z(\gamma)}\right )\ ;\ \gamma \in \Gamma \right \} .
\]
(Recall that $\phi _z$ is an involution.)  Writing 
\[
\gamma ' := \phi _z \left ( \frac{|\phi _z(\gamma)|-\delta}{1-\delta |\phi _z(\gamma)|}\frac{\phi _z(\gamma)}{|\phi _z(\gamma)}\right ),
\]
we have 
\[
|\phi _{\gamma}(\gamma')| =\left |  \phi _{\phi _z(\gamma)}\left ( \frac{|\phi _z(\gamma)|-\delta}{1-\delta |\phi _z(\gamma)|}\frac{\phi _z(\gamma)}{|\phi _z(\gamma)|}\right ) \right | = \delta |\phi _z(\gamma)|.
\]
It follows that the new sequence $\Gamma^{\delta}_z$ has the following property:  If we enumerate $\Gamma = \{\gamma _1,\gamma_2,....\}$ then there is an enumeration $\Gamma^{\delta}_z = \{ \gamma _1',\gamma _2',...\}$ such that 
\[
\sup _{i \in \N} |\phi _{\gamma _i}(\gamma _i')| \le \delta.
\]
It follows from Proposition \ref{jiggle-disk} that $\Gamma ^{\delta}_z$ is also interpolating, with interpolation constant controlled by that of $\Gamma$.  Therefore \eqref{jensen-est-disk} holds with $\Gamma ^{\delta}_z$ in place of $\Gamma$.  By changing variables in the integrals according to the transformation 
\[
\zeta \mapsto \phi _z \left ( \frac{r-\delta}{r(1-r\delta)}\phi _z(\zeta)\right )=: u,
\]
a straightforward calculation finds that for $r \sim 1$, 
\[
\int _{\phi _z (\A_r)} \log \frac{r^2}{|\phi _z(\zeta)|^2} \delta _{\Gamma} \le \int _{D_r(z)} \log \frac{r^2}{|\phi _z(\zeta)|^2}(\Delta \vp (\zeta)-2\omega _P(\zeta)) - m C_1 c_r \delta 
\]
for some positive constant $C_1$ that is independent of $z$, $r$ and $\delta$\footnote{Strictly speaking, this transformation changes the weight function $\vp$.  However, the new weight function satisfies the same hypotheses, and in particular the same curvature bounds, so the result holds for the original weight as well: see the analogous comment in \cite{v-rs1}.}.  (Here we have used that $\Delta \vp - 2\omega _P \ge m\omega _P$; the constant $c_r$ shows up because it is the hyperbolic area of $\A_r$, which in turn is asymptotic to the hyperbolic area of $D_r(z)$.)  It follows that 
\[
D^+_{\vp} (\Gamma) \le 1 - \frac{C_1m \delta }{M} < 1.  
\]
This completes the proof of Theorem \ref{disk-char}.
\qed

\section{Interpolation in $(\D ^*,\omega _P)$}\label{pdisk-section}

\subsection{Some elementary geometry of $(\D^*, \omega _P)$}

\subsubsection{Universal covers} 

We fix the universal covering map $\fp : \D \to \D^*$ defined by 
\[
\fp (x) = e^{\frac{x+1}{x-1}},
\]
and denote by $G_{\fp} \subset {\rm Aut}(\D)$ the deck group of $\fp$. We shall also have occasion to consider the upper half plane representation of the universal cover.  Letting 
\[
P :\C \to \C^*; z\mapsto e^{\ii z} \quad \text{and} \quad \h := \{ z\in \C\ ;\ \im z > 0\},
\]
we fix the covering map 
\[
P_{\h} := P|_{\h} : \h \to \D^*,
\]
and then the deck group is cyclic group generated by the translation $G(z) = z+2\pi$.  Of course,
\[
\fp (x) = P\left (\ii \tfrac{1+x}{1-x}\right ).
\]
\begin{rmk}
Observe also that $P :\C \to \C^*; z \mapsto e^{\ii z}$ is the universal cover for $\C^*$.
\red
\end{rmk}

\subsubsection{Distance and area}

We denote by $d_P$ the geodesic distance on $\D^*$ induced by the Poincar\'e metric 
\[
\omega _P = \frac{\ii dz \wedge d\bar z}{2|z|^2 (\log \frac{1}{|z|^2})^2}
\]
of $\D^*$.  One can calculate that if $\arg (z/w) = 0$ then 
\begin{equation}\label{ann-dist}
d_P(z,w) = \frac{1}{2} \left | \log \log \frac{1}{|z|^2} - \log \log \frac{1}{|w|^2} \right|,
\end{equation}
while if $|\theta - \phi| \le \pi$,  
\[
d_P\left (re^{\ii \theta}, re^{\ii \sigma}\right ) = \frac{|\theta - \sigma|}{2 \log \frac{1}{r}}.
\]
When the points $z$ and $w$ are sufficiently close together, we can find a disk in $\D^*$ that contains them both and is the biholomorphic image of a disk in $\D$, via the universal covering map.  Since the latter is a local isometry, the distance between the two points in question is the distance between their pre-images in the aforementioned disk.

Note that, at least for $z$ close to the origin, the injectivity radius of $\omega _P$ at $z$, defined to be the largest $r$ such that $\dot D _r(z)$ is contractible, is 
\[
\iota _P(z) = \frac{\pi}{2\log \frac{1}{|z|^2}}.
\]

Recall from the introduction that $\hat \iota _{\omega _P}(z) = \min (\iota _P(z), 1)$.  We can compute the area of the disk $\dot D_{\hat \iota _{\omega _P}(z)}(z)$.  Indeed, since the universal covering map $\fp :\D \to \D ^*$ is a local isometry, $\fp ^{-1}(\dot D_{\iota _P(z)}(z))$ is a disjoint union of hyperbolic disks in $\D$ of radius $\iota _P(z)$, with each disk centered at exactly one point of the sequence $\fp ^{-1}(z) \subset \D$.   Now, the Poincar\'e metric (of the unit disk) is invariant under the automorphism group.  Thus, in the notation of the introduction, 
\[
A _{\omega _P}(z) := \int _{\dot D_{\hat \iota _{\omega _P}(z)}(z)} \omega _P = \int _{|\zeta| < \tanh(\hat \iota _{\omega _P}(z))} \frac{\ii d\zeta \wedge d\bar \zeta}{2(1-|\zeta|^2)^2} = \frac{\pi (\tanh(\hat \iota _{\omega _P}(z)))^2}{1-(\tanh(\hat \iota _{\omega _P}(z)))^2}.
\]
It follows that for each $c \in (0,1)$ there exists a constant $C_c > 0$ with the property that 
\begin{equation}\label{puncture-a}
C_c^{-1} \frac{1}{(\log \frac{1}{|z|^2})^2} \le A _{\omega _P}(z) \sim C_c \frac{1}{(\log \frac{1}{|z|^2})^2}, \qquad 0 < |z| \le c,
\end{equation}
and 
\[
C_c^{-1} \le A_{\omega _P}(z) \le C_c, \qquad c < |z| < 1.
\]


\subsection{Sequences in $\D^*$}\label{pdisk-seq-section}

In view of the dichotomy of the growth rate of $A_{\omega _P}$, it is reasonable to split up the interpolation problem into two parts.  To this end, we make the following definition.

\begin{defn}
We say that a sequence $\Gamma_1 \subset \D^*$ is {\it supported near the puncture of $\D^*$} if there exists a positive constant $c<1$ such that $|\gamma| \le c$ for all $\gamma \in \Gamma_1$, and that $\Gamma_2\subset \D^*$ is {\it supported near the border of $\D$} if there exists a positive constant $c<1$ such that $|\gamma| > c$ for all $\gamma \in \Gamma_2$.
\red
\end{defn}

Note that, near the border of $\D^*$, the geometry is that of the hyperbolic unit disk.  On the other hand, near the puncture the geometry is really quite different, as we now show.

\subsubsection{Geometrically blowing up the puncture into a Euclidean cylinder}\label{cyl-geom-paragraph}

Let $\Gamma \subset \D^*$ be a closed discrete subset that is supported near the puncture of $\D^*$.  In view of the estimate \eqref{puncture-a} for $A_{\omega _P}$, the Hilbert space $\ell ^2 (\Gamma, \vp)$ is quasi-isometric to the Hilbert space 
\[
\fl ^2 (\Gamma, e^{-\vp}) := \left \{ f :\Gamma \to \C\ ;\ \sum _{\gamma \in \Gamma} \frac{|f(\gamma)|^2e^{-\vp(\gamma)}}{(\log \frac{1}{|\gamma|^2})^2} < +\infty \right \}.
\]
Thus a sequence $\Gamma$ supported near the puncture is an interpolation sequence if and only if the restriction map $\sr _{\Gamma} :\sh ^2(\D ^*, e^{-\vp}\omega _P) \to \fl ^2 (\Gamma, e^{-\vp})$ is surjective.

Now let 
\[
\psi (\zeta) := \vp (\zeta) + 2 \log \log \frac{1}{|\zeta|^2}.
\]
Then 
\[
e^{-\vp}\omega _P = e^{-\psi} \omega _c,
\]
where 
\[
\omega _c(\zeta)  = \frac{\ii d\zeta \wedge d\bar \zeta}{2 |\zeta|^2}
\]
is (the restriction to $\D^*$ of) the cylindrical metric for $\C^*$, which was studied in \cite{v-rs1}.  We therefore have 
\[
\sh ^2(\D^* , e^{-\vp} \omega _P) = \sh ^2 (\D^*, e^{-\psi}\omega _c) \quad \text{and} \quad \fl ^2 (\Gamma, e^{-\vp})  = \left \{ f :\Gamma \to \C\ ;\ \sum _{\gamma \in \Gamma} |f(\gamma)|^2e^{-\psi(\gamma)} < +\infty \right \},
\]
so the interpolation problem for sequences that are supported near the puncture is almost the same as the interpolation problem on the cylinder $(\C^*, \omega _c)$.  The latter problem was considered in \cite{v-rs1}, where we obtained necessary and sufficient conditions for interpolation.  With the cylindrical characterization in mind, and with the computation 
\[
\Delta \psi = \Delta \vp - 4 \omega _P = (\Delta \vp - 2\omega _P) - 2\omega _P,
\]
it is clear that the condition $\Delta \vp \ge (2+m)\omega _P$, which was sufficient for the case of the disk, will not work in the punctured disk.  We must ask that the weight $\vp$ further satisfy an inequality at least as strong as 
\[
\Delta \vp \ge 4\omega _P
\]
in some neighborhood of the origin in $\D^*$.  (In fact, there is a typically stronger density condition, which we will state shortly.)  

\begin{rmk}
In the cylindrical interpolation problem studied in \cite{v-rs1}, we had a somewhat stronger requirement, namely that $\Delta \psi \ge \ve \omega _c$ everywhere (in $\C^*$).  However, since we are working on $\D^*$ rather than $\C^*$, we have at our disposal both hyperbolic geometry and the technique of Donnelly-Fefferman-Ohsawa, and these will allow us to glue together extended data near the puncture and near the border of $\D^*$.
\red
\end{rmk}







\subsection{Weight averages and density of sequences supported near the border}

Since the geometry of the interpolation problem near the border is that of the hyperbolic unit disk, we begin by using hyperbolic geometry to define the density of a sequence supported near the border.

\subsubsection{Singularities along a sequence supported near the border}\label{border-sing-par}

Let $\Gamma \subset \D^*$ be a sequences supported near the border.   We can view $\Gamma$ as a subset of $\D$ and, as such, define the functions $\sigma^{\Gamma}_r: \Gamma \to \R$ and $S^{\Gamma}_r :\Gamma \to \R_+$, as well as the $(1,1)$-form $\Upsilon ^{\Gamma}_r$, as in Proposition \ref{lambda-lemma-disk}.  These functions are all obtained after choosing $T \in \co (\D)$ with ${\rm Ord}(T) = \Gamma$ and setting 
\[
\lambda ^T_r (z) := \frac{1}{c_r} \int _{\A _r} \log |T(\vp_z(\zeta))|^2 \log \frac{r^2}{|\zeta|^2} \omega^{\D} _P(\zeta),
\]
where $c_r$ is defined by \eqref{cr-defn}.  Of course, Proposition \ref{lambda-lemma-disk} applies to these objects.

In order to indicate that we are working near the border of $\D^*$, rather than in $\D$, we shall write
\[
\sigma ^{b,\Gamma}_r:= \sigma ^{\Gamma}_r, \quad S^{b,\Gamma}_r:=\sigma ^{\Gamma}_r \quad \text{and} \quad \Upsilon^{b,\Gamma}_r:=\Upsilon^{\Gamma}_r.
\]
We emphasize that, even though we are studying a problem on $\D^*$, we are using functions $T$ that are holomorphic across the origin.

\subsubsection{Logarithmic means}\label{border-means-par}

Let $\vp \in \sC ^2 (\D^*)$ be a smooth weight function.  The function $\vp$ is not necessarily smooth across the puncture, so it is necessary to modify it near the puncture.  This is easily done as follows:  let $h_c : [0,1] \to [0,1]$ be a smooth, increasing function such that $h|_{[0,c/2]}\equiv 0$ and $h|_{[c,1]}\equiv 1$, where $c\in (0,1)$.  Define the $c$-truncated logarithmic mean of $\vp$ as 
\[
\vp _{r,c} (z) := (h_c \vp)_r = \frac{1}{a_r} \int _{D_r(z)} h_c(|\zeta|)\vp (\zeta) \log \frac{r^2}{|\phi _z (\zeta)|^2} \omega _P^{\D}(\zeta).
\]
Note that if, in the set $0 < c/2 \le |z| <1$ the weight function $\vp$ satisfies the curvature estimates  
\[
- M \omega _P \le \Delta \vp \le M \omega _P,
\]
then 
\[
- M \omega _P \le \Delta \vp_{c,r} \le M \omega _P \quad \text{ in all of }\D.
\]
It follows from Proposition \ref{mean-wt-hilb} that $|h_c(z)\vp(z) - \vp _{c,r}(z)| \le \tilde C_r$.  In particular, we have the estimate 
\[
|\vp (z) - \vp _{c,r}(z)| \le C_r \quad \text{for }c\le |z| < 1,
\]
where $C_r$ is independent of $z$ (though it does depend on $c$ and the constant $M$).

\subsubsection{Density of a closed discrete subset supported near the border of $\D^*$}\label{bord-density-defn}

Let $\Gamma \subset \D^*$ be a sequence supported near the border.   Viewing $\Gamma$ as a sequence in $\D$, we define the upper density of $\Gamma$ to be 
\[
D^{b+}_{\vp} (\Gamma) := \inf \left \{ \alpha \ ;\ \forall\ r _o \in (c,1)\ \exists\ r\in (r_o,1) \text{ such that }\alpha \Delta \vp _{c,r} \ge \Upsilon ^{\Gamma}_r \right \}.
\]
Since the sequence is supported near the border, the inequality $\alpha \Delta \vp _{c,r} \ge \Upsilon ^{\Gamma}_r$ will hold if it holds near the border of $\D^*$.  Thus the definition of $D^{b+}_{\vp}(\Gamma)$ is independent of the choice of $c \in (0,1)$.

\begin{rmk}
If $\vp$ is a smooth weight in $\D$ then $D^{b+}_{\vp}(\Gamma)$ is just the upper density $D^+_{\vp}(\Gamma)$ of the sequence $\Gamma$ defined in (ii) of Theorem \ref{disk-char}.
\red
\end{rmk}

\subsection{Weight averages and density of sequences supported near the puncture}

In view of Paragraph \ref{cyl-geom-paragraph}, the geometry of the interpolation problem near the puncture is cylindrical.  We therefore adapt the ideas of \cite{v-rs1} to define weight averages and density of sequences supported near the puncture.

\subsubsection{Singularities along a sequence supported near the puncture}\label{p-sing-par}

For a locally integrable function $h$ in the Euclidean annulus $\A_r(z)$ of inner radius $1$ and outer radius $r>1$, and center $z \in \C$, we define 
\[
A^h_r (z) := \frac{1}{c_r} \int _{\A_r(z)}  h(\zeta) \log \frac{r^2}{|\zeta -z|^2}  \omega _o (\zeta) =  \frac{1}{c_r} \int _{\A_r(0)}  h(z-\zeta) \log \frac{r^2}{|\zeta|^2}  \omega _o (\zeta),
\]
where $\omega _o$ denotes the Euclidean metric in $\C$.  Note that if $h$ is $2\pi$-periodic, then so is $A^h_r$.

Let $\Gamma \subset \D^*$ be a closed discrete subset supported near the puncture.  There exists $T \in \co (\C^*) \subset \co (\D^*)$ such that ${\rm Ord}(T) = \Gamma$.  As in \cite{v-rs1}, we define the covered logarithmic mean $\check \lambda ^T _r : \C^* \to \R$ of $\log |T|^2$ by 
\[
P ^* \check \lambda ^T _r:= A^{P^*(\log |T|^2)}_r.
\]
Evidently the function $A^{P^*(\log |T|^2)}_r$ is $2\pi$-periodic, so $\check \lambda ^T _r$ is well-defined. 
\begin{defn}\label{b-singularity}
The function $\check \lambda ^T_r$ is called a {\it potential function} for $\Gamma$.
\red
\end{defn}

The following proposition was proved in \cite{v-rs1}.

\begin{prop}\label{lambda-lemma-pdisk}
Let $\Gamma \subset \D^*$ be a closed discrete subset that is supported near the puncture, and choose $T\in \co (\C^*)$ such that ${\rm Ord}(T) = \Gamma$.
\begin{enumerate}
\item[(a)] The functions $\sigma _r :\C^* \to [0,\infty)$ and $S_r :\Gamma \to (0,\infty)$ defined by 
\[
\sigma^{\Gamma} _r (z) = |T(z)|^2 e^{-\check\lambda _r^T(z)} \quad \text{and} \quad S^{\Gamma}_r (\gamma) = |dT(\gamma)|^2_{\omega _P}e^{-\check\lambda _r^T(\gamma)},
\]
and the $(1,1)$-form 
\[
\Upsilon ^{\Gamma} _r := \Delta \check \lambda _r ^T
\]
are independent of the choice of $T$.  
\item[(b)] $\sigma^{\Gamma} _r \le 1$.
\item[(c)] For any $\tilde \gamma \in \tilde \Gamma := P^{-1}(\Gamma)$, any $r$ such that $\tilde \Gamma \cap D^o _r (\tilde \gamma) = \{\tilde \gamma\}$, and $\ve \in (0,r)$, and any $z \in D^o_{r}(\tilde \gamma)$ such that $|\tilde \gamma- z| > \ve$, we have the estimate 
\[
\sigma ^{\Gamma} _r (P(z)) \ge C_r \ve ^2.
\]
On the other hand, $\frac{1}{\sigma ^{\Gamma}_r}$ is not locally integrable in any neighborhood of any point of $\Gamma$.
\item[(d)] One has the formula
\begin{equation}\label{density-formula} 
\frac{\Delta \check \lambda ^{T} _r (z)}{2\omega _c(z)} =  \frac{1}{c_r} \sum _{\tilde \gamma \in P^{-1}(\Gamma) \cap \A ^o_r(q)} \log \frac{r^2}{|\tilde \gamma-q|^2},
\end{equation}
where $q \in P^{-1} (z)$ is any point.
\end{enumerate}
\end{prop}

In order to indicate that we are working near the puncture of $\D^*$, rather than in $\C^*$, we shall write
\[
\sigma ^{*,\Gamma}_r:= \sigma ^{\Gamma}_r, \quad S^{*,\Gamma}_r:=\sigma ^{\Gamma}_r \quad \text{and} \quad \Upsilon^{*,\Gamma}_r:=\Upsilon^{\Gamma}_r.
\]

\subsubsection{$\ve$-extended Logarithmic means and $\ve$-extended covered means}\label{puncture-means-par}

Let $\tau$ be a locally integrable weight function on $\h$.  For a given small positive number $\ve$, the function 
\[
\tau _{\ve}(z) = \tau (z+\ve)
\]
is defined on the closure of $\h$ in $\C$.  We then define 
\[
\tau ^+_{\ve}(z) := \left \{
\begin{array}{c@{\qquad}l}
\tau _{\ve}(z), & z\in \h \\
\tau _{\ve}(\bar z), & z \in \C - \h
\end{array}
.
\right .
\]
Observe that if $a\omega _o \le \Delta \tau \le b\omega _o$ for any $a, b \in \R$, then the same is true for $\tau _{\ve} ^+$ (as a current).  In particular, if $\tau$ is subharmonic then so is $\tau ^+_{\ve}$.  We can then define the logarithmic mean of $\tau _{\ve}^+$ as was done in \cite{v-rs1}:
\[
\tau ^+_{\ve,r} (z) := \int _{D^o_r(0)} \tau^+ _{\ve}(z-\ve-\zeta) \log \frac{r^2}{|\zeta|^2} \omega _o(\zeta).
\]
Observe that, for $\im z > r +\ve$, we have 
\[
\Delta \tau ^+_{\ve,r} (z) = \left ( \int _{D^o_r(0)} \frac{\di ^2 \tau}{\di \zeta \di \bar \zeta} (z-\zeta) \log \frac{r^2}{|\zeta|^2} \omega _o(\zeta)\right )2 \omega _o(z) =\left ( \int _{D^o_r(z)} \frac{\di ^2 \tau}{\di \zeta \di \bar \zeta}(\zeta) \log \frac{r^2}{|z-\zeta|^2} \omega _o (\zeta)\right ) 2\omega _o(z).
\]

Now suppose $\tau$ is $2\pi$-periodic in $\h$, i.e., 
\[
\tau (z+2\pi) = \tau (z), \quad z \in \h.
\]
Then $\tau ^+_{\ve}$ is $2\pi$-periodic in $\C$, and therefore so is $\tau^+_{\ve,r}$.

Finally, if $\psi$ is a locally integrable function in $\D^*$, then $\tilde \psi := P_{\h}^*\psi$ is $2\pi$-periodic in $\h$, and thus $\tilde \psi ^+_{\ve, r}$ is $2\pi$-periodic in $\C$.  Therefore there exists a locally integrable function $\mu _{\ve, r}(\psi)$ on $\C^*$ such that 
\[
P^*\mu _{\ve, r}(\psi) = \tilde \psi ^+_{\ve, r}.
\]
\begin{defn}
The function $\mu _{\ve,r}(\psi)$ is called the $\ve$-extended covered mean of $\psi$.
\red
\end{defn}

\begin{rmk}\label{log-mean-in-half-space}
Note that 
\[
|z| < e^{-(r+\ve)}\ \Rightarrow \ \mu _{\ve,r}(\psi) (z)= \mu _{r}(\psi)(z),
\]
where $\mu _{r}(\Psi)$ is the covered mean of a function $\Psi$ on $\C^*$, as defined in \cite{v-rs1}.  (The definition of $\mu _r$ is local on $\C^*$, so it makes sense to talk about $\mu _r(\Psi)(z)$ for a function $\Psi$ that is only defined in $\D^*$ and not on all of $\C^*$, so long as the point $z \in \D^*$ is sufficiently close to the origin.)
\red
\end{rmk}

\subsubsection{Density of a closed discrete subset supported near the puncture}\label{dstar-density-definitions}

Finally, we introduce the following definition.

\begin{defn}\label{punc-density-defn}
Let $\Gamma \subset \D^*$ be a closed discrete subset supported near the puncture, let $\vp \in L^1_{\ell oc} (\D^*)$ be an upper semi-continuous weight function satisfying 
\[
\Delta \vp \ge 4\omega _P,
\]
and let $\mu _{\ve,r}(\psi)$ be the $\ve$-shifted covered mean of $\psi$.The number 
\[
D ^{*+}_{\vp} (\Gamma)  = \inf \left \{ \alpha \ ; \forall r _o \in (0,\infty)\ \exists \ r > r_o \text{ such that }\alpha \ii \di \dbar \mu _{\ve,r} (\vp + 2 \log \log \tfrac{1}{|z|^2})  \ge  \Upsilon ^{\Gamma}_r \right \}
\]
is called the {\it puncture density} of $\Gamma$.
\red
\end{defn}

If one unwinds the definitions, one sees that the puncture density of $\Gamma$ with respect to $\vp$ is simply the density of $P^{-1}(\Gamma)$  in $\C$, with respect to $(\vp + 2\log \log \frac{1}{|z|^2})^+ _{\ve}$.  Since the density of a sequence supported near the puncture is computed using disks whose center converges to the puncture, the definition of $D^{*+}_{\psi}(\Gamma)$ is independent of the choice of $\ve > 0$.  Thus if the sequence $\Gamma \subset \D^*$ is supported near the puncture, then $D^{*+}_{\psi}(\Gamma)$ is just the cover density of $\Gamma \subset \C^*$, computed with respect to the weight $\psi := \vp + 2 \log \log \frac{1}{|z|^2}$, in the sense of \cite{v-rs1}. 

\subsection{Statement of the interpolation theorem in $(\D^*,\omega_P)$}

\subsubsection{Density and uniform separation of sequences}

We have seen that the geometry of the interpolation problem near the puncture is cylindrical, whereas near the border, it is hyperbolic.  As such, we define density and uniform separation in terms of these geometries.

\begin{defn}\label{unif-sep-density-defn}
For a closed discrete subset $\Gamma \subset \D^*$ and a number $a \in (0,1)$, write 
\[
\Gamma ^* _a := \{ \gamma \in \Gamma\ ;\ |\gamma| \le a \} \quad \text{ and } \quad \Gamma ^b _a :=  \{ \gamma \in \Gamma\ ;\ |\gamma| > a\}.
\]
\begin{enumerate}
\item[(i)] For a weight function $\vp \in \sC ^2(\D^*)$ satisfying the curvature inequalities
\begin{enumerate}
\item[(a)] $\Delta \vp -2 \omega _P\ge 0$, and 
\item[(b)] there exists $c \in (0,1)$ such that, for all $0 < |\zeta| < c$, $\Delta \vp (\zeta) \ge 4 \omega _P(\zeta)$, 
\end{enumerate}
the upper density of $\Gamma$ is the number 
\[
\dot D^+_{\vp}(\Gamma) := \inf _{a \in (0,1)} \max (D^{b+}_{\vp}(\Gamma^b_a), D^{*+}_{\vp}(\Gamma^*_a)),
\]
where $D^{b+}_{\vp}(\Gamma)$ and $D^{*+}_{\vp}(\Gamma)$ are as in Paragraph \ref{bord-density-defn} and Definition \ref{punc-density-defn} respectively.

\item[(ii)] The sequence $\Gamma \subset \D^*$  is said to be uniformly separated if  the non-negative numbers
\[
R^*_{\Gamma} := \frac{1}{2} \inf \{ d_c (\gamma, \mu)\ ;\ \gamma, \mu \in \Gamma ^*_a,\ \gamma \neq \mu\} \quad \text{and} \quad R^b _{\Gamma} := \frac{1}{2} \inf \{ d_P (\gamma, \mu)\ ;\ \gamma, \mu \in \Gamma ^b_a,\ \gamma \neq \mu\}
\]
are positive.    
\red
\end{enumerate}
\end{defn}

\noi Here 
\[
d_c(z,w):= |\log z - \log w| = \sqrt{(\log |z/w|)^2 +(\arg z - \arg w)^2}
\]
is the cylindrical distance, i.e., the geodesic distance for $\omega _c$ in $\C^*$.

\begin{rmk}
Clearly the uniform separation of $\Gamma$ is independent of $a$.  In fact, because the densities are realized asymptotically near the boundary of $\D^*$, the number $\dot D^+_{\vp}(\Gamma)$ is also independent of $a$.
\red
\end{rmk}

Finally, we have the following lemma.

\begin{lem}\label{ot-and-unif-sep-pdisk}
Let $\Gamma \subset \D^*$ be a uniformly separated sequence that is supported near the puncture, and let $T \in \co (\D^*)$ be a function such that ${\rm Ord}(T)=\Gamma$.  Then, with $\tilde \lambda^T_r$ denoting the function defined in Paragraph \ref{p-sing-par}, the sequence $\Gamma$ is uniformly separated if and only if for each $r > 0$ there exists $C_r>0$ such that 
\[
\sup _{\Gamma} |dT|^2e^{-\tilde \lambda ^T_r} \ge C_r.
\]
\end{lem}

The proof of Lemma \ref{ot-and-unif-sep-pdisk} is the same as that of Proposition 2.7 in \cite{v-rs1}.

\subsubsection{Statement of the interpolation theorem}

\begin{d-thm}\label{pdisk-char}
Let $\vp \in \sC ^2 (\D^*)$ be a weight function satisfying the following conditions: there exists a constant $c \in(0,1)$, and positive constants $m$ and $M$, such that 
\begin{enumerate}
\item[(B)] for all $c \le |\zeta| <1$, $m \omega _P(\zeta) \le \Delta \vp(\zeta) -2 \omega _P(\zeta) \le M \omega _P(\zeta)$, and 
\item[($\star$)] for all $0 < |\zeta| < c$, $m \omega _c(\zeta) \le \Delta \vp (\zeta) -  4 \omega _P(\zeta) \le M \omega _c(\zeta)$.
\end{enumerate}
Let $\Gamma \subset \D^*$ be a closed discrete subset.  Then the restriction map 
\[
\sr _{\Gamma} : \sh ^2 (\D^*, e^{-\vp} \omega _P) \to \ell ^2 (\Gamma , e^{-\vp})
\]
is surjective if  
\begin{enumerate}
\item[(i+)] $\Gamma$ is uniformly separated, and 
\item[(ii+)] $\dot D^+_{\vp}(\Gamma) < 1$.
\end{enumerate}
Conversely, if $\sr _{\Gamma}$ is surjective then 
\begin{enumerate}
\item[(i-)] $\Gamma$ is uniformly separated, and 
\item[(ii-)] $\dot D^{+}_{\vp}(\Gamma) \le 1$.  More precisely, $D^{b+}_{\vp}(\Gamma) < 1$ and $D^{*+}_{\vp}(\Gamma) \le 1$.
\end{enumerate}
\end{d-thm}

\subsection{Sufficiency}

As in the case of the unit disk, we begin with the proof that Conditions (i+) and (ii+) of Theorem \ref{pdisk-char} are sufficient to guarantee the surjectivity of the restriction map 
\[
\sr _{\Gamma}:\sh ^2(\D^*, e^{-\vp} \omega _P) \to \ell ^2 (\Gamma, e^{-\vp}).
\]

\begin{d-thm}\label{pdisk-suff}
Let $\vp \in \sC ^2 (\D^*)$ satisfy conditions {\rm (B)} and $(\star)$.  Let $\Gamma \subset \D^*$ be uniformly separated, and assume $\dot D^+_{\vp} (\Gamma)  < 1$.  Then the restriction map $\sr _{\Gamma} : \sh ^2 (\D^*, e^{-\vp} \omega _P) \to \ell ^2 (\Gamma , e^{-\vp})$ is surjective.
\end{d-thm}

\subsubsection{Strong sufficiency} 
In fact, we are going to prove a somewhat stronger result.

\begin{d-thm}\label{pdisk-ssuff}
Let $\vp$ be a subharmonic weight function on $\D^*$ satisfying 
\[
\Delta \vp -2 \omega _P\ge m \omega _P, \quad \text{and} \quad \Delta \vp (\zeta) - 4\omega _P(\zeta) \ge m \omega _c(\zeta)
\]
for some positive constant $m$, with the second inequality holding for $0 < |\zeta| < c$ for some $c\in (0,1)$.  Let $\Gamma \subset \D^*$ be uniformly separated, and write $\Gamma = \Gamma _1\cup \Gamma_2$, where $|\gamma| > c$ for $\gamma \in \Gamma _1$ and $|\gamma |< c$ for $\gamma \in \Gamma _2$.  Suppose there exist $\ve >0$ and $r \in (0,1)$ such that 
\[
\Delta \vp - 2\omega _P \ge (1+\ve) \Upsilon ^{b,\Gamma_1}_r \quad \text{in } \D^*,
\]
and 
\[
\Delta \vp - 4\omega _P \ge (1+\ve) \Upsilon ^{*,\Gamma_2}_r \quad \text{on the set }\{z \in \C\ ; \ 0 < |z| < c\}.
\]
Then the restriction map $\sr _{\Gamma} : \sh ^2 (\D^*, e^{-\vp} \omega _P) \to \ell ^2 (\Gamma , e^{-\vp})$ is surjective.
\end{d-thm}

Theorem \ref{pdisk-ssuff} is proved in three parts.  In the first part, one solves the interpolation problem for sequences supported near the border.  In the second part, one interpolates the data supported near the puncture with a function that is holomorphic across the puncture, i.e., a positive distance away from the Border.  This second step follows from the work in \cite{v-rs1}.  Finally these two interpolation functions are glued together to complete the proof.

\begin{d-thm}[Strong sufficiency: border case]\label{pdisk-strong-suff-b}
Let $\vp \in L ^1_{\ell oc} (\D^*)$ be a subharmonic weight function satisfying 
\[
\Delta \vp  -2 \omega _P \ge m \omega _P
\]
for some positive constant $m$.  Let $\Gamma \subset \D^*$ be uniformly separated and supported near the border, and assume 
\[
\Delta \vp - 2\omega _P \ge (1+\ve) \Upsilon ^{b,\Gamma}_r
\]
for some positive number $\ve$.  Then the restriction map $\sr _{\Gamma} : \sh ^2 (\D^*, e^{-\vp} \omega _P) \to \ell ^2 (\Gamma , e^{-\vp})$ is surjective.
\end{d-thm}

\begin{proof}
Since $\Gamma$ is supported near the border, there exists a constant $\delta > 0$ such that the pseudohyperbolic disks $\{ D_{\delta}(\gamma)\ ;\ \gamma \in \Gamma\}$ are pairwise-disjoint and contractible.  By Lemma \ref{loc-extn-disk} there exist holomorphic functions $\{ g_{\gamma} \in \co (D_{\delta}(\gamma)\ ;\ \gamma \in \Gamma\}$ such that 
\[
g_{\gamma}(\gamma) = e^{\vp(\gamma)/2} \quad \text{and} \quad \int _{D_{\delta}(\gamma)}|g_{\gamma}|^2e^{-\vp} \omega _P \le C_{\delta}, \qquad \gamma \in \Gamma,
\]
where $C_{\delta}$ is a universal constant, and in particular,  independent of $\gamma$.  (Note that Lemma \ref{loc-extn-disk} is formulated in the unit disk, but since the sequence $\Gamma$ is supported near the border, the same proof works for the punctured disk.)  

Fix a datum $f \in \ell ^2(\Gamma, e^{-\vp})$ to be extended, i.e., 
\[
\sum _{\gamma \in \Gamma} |f(\gamma)|^2e^{-\vp(\gamma)} < +\infty.
\]
Let $\chi \in \sC ^{\infty}(\R)$ be a decreasing function satisfying
\[
\chi(x) = 1 \text{ for } x \le \tfrac{1}{2}, \ \chi (x) = 0 \text{ for }x\ge 1, \text{ and } |\chi '(x)| \le 3 \text{ for all }x.
\]
Consider the function 
\[
\tilde F(\zeta) := \sum _{\gamma \in \Gamma} f(\gamma)e^{-\vp(\gamma)/2} g_{\gamma}(\zeta) \chi \left ( \delta ^{-2} |\phi _{\gamma}(\zeta)|^2\right ).
\]
Then 
\[
\tilde F \in \sC ^{\infty} (\D^*) \quad \text{and} \quad \tilde F|_{\Gamma} = f.
\]
By Lemma \ref{loc-extn-disk}, 
\[
\int _{\D^*} |\tilde F|^2e^{-\vp} \omega _P \le \sum _{\gamma \in \Gamma} |f(\gamma)|^2 e^{-\vp(\gamma)}\int _{D_{\delta}(\gamma)} |g_{\gamma}|^2e^{-\vp} \omega _P \le C_{\delta} ||f||^2,
\]
so that $\tilde F \in L^2(\D^*, e^{-\vp} \omega _P)$.  We now correct $\tilde F$ to be holomorphic and still interpolate $f$. Thus we seek a solution $u$ of the equation 
\[
\dbar u = \alpha := \dbar \tilde F
\]
that lies in $L^2(\D, e^{-\vp}\omega _P)$ and vanishes along $\Gamma$.  To this end, consider the weight function 
\[
\psi  = \vp + \log \sigma _r ^{b,\Gamma}.
\]
Note that $\chi '(x) = 0$ for $|x|\le 1/2$ and that 
\[
|d \phi _z|^2_{\omega _P} = (1-|\phi _z|^2)^2\le 1.
\]
We compute that
\[
\alpha(\zeta) = \sum _{\gamma \in \Gamma} f(\gamma)e^{-\vp(\gamma)/2} g_{\gamma}(\zeta) \chi '\left ( \delta^{-2} |\phi _{\gamma}(\zeta)|^2\right )\frac{\phi _z (\zeta) \overline{\phi _z'(\zeta)}}{\delta^2}d\bar \zeta.
\]
Thus by Property (c) of Proposition \ref{lambda-lemma-disk},
\begin{eqnarray*}
\int _{\D^*} |\alpha|^2_{\omega _P} e^{-\psi} \omega _P &=& \sum _{\gamma \in \Gamma} |f(\gamma)|^2e^{-\vp(\gamma)}\int _{D_{\delta}(\gamma)} |g_{\gamma}(\zeta)|^2 e^{-\psi(\zeta)} \left | \chi '\left ( \delta^{-2} |\phi _{\gamma}(\zeta)|^2\right )\right |^2 \frac{|\phi _z (\zeta)|^2 |d\phi _z(\zeta)|^2_{\omega_P}}{\delta^4}\omega _P\\
&\lesssim &  \sum _{\gamma \in \Gamma} |f(\gamma)|^2e^{-\vp(\gamma)}\int _{D_{\delta}(\gamma)} |g_{\gamma}(\zeta)|^2 e^{-\vp(\zeta)} \omega _P < +\infty.
\end{eqnarray*}
Since 
\[
\Delta \psi - 2 \omega _P \ge \frac{1}{1+\ve}\left ( \Delta \vp - 2\omega _P - (1+\ve)\Upsilon ^{\Gamma} _r\right ) + \frac{\ve}{1+\ve} \left ( \Delta \vp -2\omega _P \right ) \ge \frac{m\ve}{1+\ve} \omega _P,
\]
Theorem \ref{punctured-disk-odf-thm} gives us a function $u$ such that 
\[
\dbar u = \alpha \quad \text{and} \quad \int _{\D^*} |u|^2e^{-\psi} \omega _P < +\infty.
\]
By the smoothness of $\alpha$ and the interior ellipticity of $\dbar$, $u$ is smooth.  By Property (c) of Proposition \ref{lambda-lemma-disk}, in particular the non-integrability of $e^{-\psi}$ along $\Gamma$, $u|_{\Gamma} \equiv 0$.  Furthermore, by (b) of Proposition \ref{lambda-lemma-disk}, 
\[
\int_{\D^*} |u|^2 e^{-\vp} \omega _P \le \int _{\D^*}|u|^2e^{-\psi} \omega _P < +\infty.
\]
Finally, set $F = \tilde F- u$.  Then $F \in \co (\D^*)$, $F|_{\Gamma} = f$, and
\[
\int _{\D^*}|F|^2e^{-\vp} \omega _P \le 2 \left ( \int _{\D^*}|\tilde F|^2e^{-\vp} \omega _P + \int _{\D^*}|u|^2e^{-\vp} \omega _P \right ) < +\infty.
\]
This completes the proof of Theorem \ref{pdisk-strong-suff-b}.
\end{proof}

\begin{d-thm}[Strong sufficiency: puncture case]\label{pdisk-strong-suff-p}
Let $\vp \in L ^1_{\ell oc} (\D^*)$ be a subharmonic weight function.  Let $\Gamma \subset \D^*$ be uniformly separated and supported near the puncture, and assume 
\[
\Delta \vp - 4\omega _P \ge (1+\ve) \Upsilon ^{*\Gamma}_r
\]
for some positive number $\ve$.  Fix $c \in (0,1)$ such that $\Gamma \subset \{ 0 < |z| < c\}$.  Then for each $f :\Gamma \to \C$ satisfying 
\[
\sum _{\gamma \in \Gamma} |f(\gamma)|^2\frac{e^{-\vp(\gamma)}}{(\log \frac{1}{|\gamma|^2})^2} <+\infty
\]
there exists $F \in \co (\{ 0 < |z| < \frac{c+1}{2} \})$ such that 
\[
F|_{\Gamma}= f \quad \text{and} \quad \int_{\{ 0 < |z| < \frac{c+1}{2} \}} |F|^2e^{-\vp} \omega _P < +\infty.
\]
\end{d-thm}

\begin{proof}
Let $f \in \ell ^2 (\Gamma , e^{-\vp})$ be the datum to be extended. By Lemma \ref{ot-and-unif-sep-pdisk}, 
\[
\sum _{\gamma \in \Gamma} |f(\gamma)|^2\frac{e^{-\vp(\gamma)}}{(\log \frac{1}{|\gamma|^2})^2} \frac{1}{|dT(\gamma)|^2e^{-\tilde \lambda ^T_r(\gamma)}} < +\infty.
\]
We are going to use the $L^2$ Extension Theorem \ref{ot-basic}.  In the notation of that theorem, we choose the data $X = \{ z \in \C\ ;\ 0 < |z| < \frac{c+1}{2}\}$, $\psi = \vp +2 \log \log \frac{1}{|z|^2}$, $\lambda = \tilde \lambda ^T_r$, and $\omega = \omega _c$, the restriction to $X$ of the cylindrical metric in $\C^*$.  Thus by $L^2$ Extension Theorem  there exists $F \in \co (X)$ such that 
\[
F|_{\Gamma}= f \quad \text{and} \quad \int _{X} |F|^2e^{-\vp} \omega _P < +\infty.
\]
This completes the proof.
\end{proof}

\begin{proof}[Proof of Theorem \ref{pdisk-ssuff}]

Let $f \in \ell ^2(\Gamma, e^{-\vp})$ be the datum to be extended.  Choose $\beta >0$ such that, with $W_{r,\delta} := \{\zeta \in \C\ ;\ r-\delta \le |\zeta|\le r+\delta \}$, 
\[
\Gamma \cap W_{c, \beta} = \emptyset.
\]
We set 
\[
\Gamma _b := \{ \gamma \in \Gamma\ ;\ |\gamma| > c\} \quad \text{and} \quad \Gamma _* := \{ \gamma \in \Gamma\ ;\ |\gamma| < c\},
\]
which are supported near the border and puncture respectively.  Associated to these sequences, we have functions $\sigma ^{b,\Gamma _b}_r$ and $\sigma ^{*, \Gamma _*}_r$, and $(1,1)$-forms 

Fix $\chi \in \sC^{\infty}_o(W_{c,\beta})$ such that $0 \le \chi \le 1$, $\chi |_{W_{c,\beta/2}} \equiv 1$, and $|\chi '| \le \frac{3}{\beta}$.  We define the function 
\[
\eta ^{\Gamma} := \left \{ 
\begin{array}{l@{\quad}l}
(1-\chi(\zeta))  \log \sigma ^{*,\Gamma _*}_r (\zeta) +C (|\zeta|^2-1)& |\zeta|\le c\\
(1-\chi(\zeta)) \log \sigma ^{b,\Gamma _b}_r (\zeta) +C (|\zeta|^2-1)& |\zeta|\ge c
\end{array}
\right . .
\]
Then for $C$ sufficiently large, 
\begin{enumerate}
\item[(i)] $e^{-\eta ^{\Gamma}}$ is not locally integrable at any point of $\Gamma$, but is smooth everywhere else, 
\item[(ii)] $\eta ^{\Gamma} \le 0$ on $\D^*$, and 
\item[(ii)] there is a continuous, positive $(1,1)$-form $\theta$ on $\C$ such that $\theta \le C\omega _o$ and 
\[
\Delta \eta ^{\Gamma} =  \theta - \mathbf{1}_{\{|\cdot |<c\}}\left (\Upsilon ^{*,\Gamma_*}_r \right )  - \mathbf{1}_{\{|\cdot |>c\}}\Upsilon ^{b,\Gamma_b}_r.
\]
\end{enumerate}

Now, by Theorems \ref{pdisk-strong-suff-b} and \ref{pdisk-strong-suff-p}, there exist 
\[
F^b \in \sh ^2(\D^*, e^{-\vp}\omega _P) \quad\text{ and }\quad F^* \in \sh ^2(\{ 0 < |z|< c\}, e^{-\vp}\omega _P)
\]
such that 
\[
F^b|_{\Gamma ^b} = f|_{\Gamma ^b} \quad \text{and} \quad F^*|_{\Gamma ^*} = f|_{\Gamma ^*}.
\]
Let 
\[
\tilde F := \mathbf{1}_{\{|\cdot |< c\}} (1-\chi)F^* + \mathbf{1}_{\{|\cdot |> c \}}(1-\chi)F^b.
\]
Then $\tilde F \in L^2(\D^*, e^{-\vp}\omega _P)$ is smooth, holomorphic in $\D^* - W_{c,\beta}$, and satisfies 
\[
\tilde F|_{\Gamma} = f.
\]
The $(0,1)$-form $\alpha := \dbar \tilde F$ is smooth, and supported in $W_{c,\beta}$.  Thus, since $F^*$ and $F^b$ are square integrable on their domains, 
\[
\int _{\D^*} |\alpha|^2_{\omega_P} e^{-\vp-\eta ^{\Gamma}}\omega _P < +\infty.
\]
Now, 
\begin{eqnarray*}
\Delta (\vp +\eta ^{\Gamma}) -2\omega _P &=&  \frac{1}{1+\ve}\left ( \Delta \vp - 2\omega _P - \left (  \mathbf{1}_{\{|\cdot |<c\}}\left (2\omega _P+ \Upsilon ^{*,\Gamma_*}_r  \right )+  \mathbf{1}_{\{|\cdot |>c\}}\Upsilon ^{b,\Gamma_b}_r  \right ) \right ) \\
&& \qquad + \frac{2}{1+\ve} \cdot \mathbf{1}_{\{|\cdot |<c\}}\omega _P + \theta + \frac{\ve}{1+\ve}(\Delta \vp - 2\omega _P)\\
&\ge & \frac{m\ve}{1+\ve} \omega _P.
\end{eqnarray*}
By Theorem \ref{punctured-disk-odf-thm} and the interior elliptic regularity of $\dbar$ there exists $u \in L^1_{\ell oc} (\D^*, e^{-\vp-\eta ^{\Gamma}}\omega _P) \cap \sC ^{\infty}(\D^*)$ such that 
\[
\dbar u = \alpha \quad \text{and} \quad \int _{\D^*} |u|^2 e^{-\vp} \omega _P \le \int _{\D^*} |u|^2 e^{-\vp-\eta ^{\Gamma}} \omega _P < +\infty.
\]
It follows that $u|_{\Gamma}\equiv 0$, and that, with $F := \tilde F- u$, $F|_{\Gamma} = f$.  Finally, 
\[
\int _{\D^*} |F|^2e^{-\vp} \omega _P \le 2 \int _{\D^*} |\tilde F|^2e^{-\vp}\omega _P + \int _{\D^*} |u|^2e^{-\vp} \omega _P < +\infty.
\]
The proof of Theorem \ref{pdisk-ssuff} is complete.
\end{proof}

\begin{proof}[Proof of Theorem \ref{pdisk-suff}]
The conditions on the weight $\vp$ allows us to apply the construction of Paragraph \ref{border-means-par} to the weight $\vp$ to obtain a weight that has the same asymptotic growth as $\vp$ near the border of $\D^*$.

Near the puncture, we apply the construction of Paragraph \ref{puncture-means-par} to the weight $\vp$ to obtain a weight that has the same asymptotic growth as $\vp$ near the puncture.  (This is the case because $\omega _P$ is locally finite near the origin.)  

We therefore have two regularized weights, $\vp_1$ and $\vp_2$, with the same asymptotics as $\vp$ near the border and puncture respectively.  Now take a function $\chi \in \sC ^{\infty}(\D)$ such that $0 \le \chi \le 1$, $\chi (z) \equiv 1$ for $|z| \le c$, and $\chi(z) \equiv 0$ for $|z| \ge \frac{c+1}{2}$.  Consider the weight function 
\[
\psi = (1-\chi) \vp _1 + \chi \vp _2.
\]
Then $\sh ^2 (\D ^*, e^{-\vp}\omega _P)$ and $\sh ^2 (\D ^*, e^{-\psi}\omega _P)$ are quasi-isomorphic Hilbert spaces, as are $\ell ^2 (\D^*, e^{-\vp})$ and $\ell ^2 (\D^*, e^{-\psi})$. 

Now, the dentsity conditions for $\vp$ imply that $\psi$ satisfies the hypotheses of Theorem \ref{pdisk-ssuff}, and thus  the restriction map 
\[
\sr _{\Gamma} : \sh ^2 (\D ^*, e^{-\psi}\omega _P) \to \ell ^2 (\D^*, e^{-\psi})
\]
is surjective.  Since the identity maps 
\[
\sh ^2 (\D ^*, e^{-\vp}\omega _P) \to \sh ^2 (\D ^*, e^{-\psi}\omega _P) \quad \text{and} \quad \ell ^2 (\D^*, e^{-\vp}) \to \ell ^2 (\D^*, e^{-\psi})
\]
are bounded linear isomorphisms, the restriction map 
\[
\sr _{\Gamma} : \sh ^2 (\D ^*, e^{-\vp}\omega _P) \to  \ell ^2 (\D^*, e^{-\vp})
\]
is also surjective.  This completes the proof of Theorem \ref{pdisk-suff}
\end{proof}

\subsection{Necessity}

We shall now state and prove the converse of Theorem \ref{pdisk-suff}.  

\begin{d-thm}\label{pdisk-nec}
Let $\vp \in \sC ^2 (\D^*)$ satisfy conditions {\rm (B)} and $(\star)$.  Let $\Gamma \subset \D^*$ be a closed discrete subset, and assume that the restriction map $\sr _{\Gamma} : \sh ^2 (\D^*, e^{-\vp} \omega _P) \to \ell ^2 (\Gamma , e^{-\vp})$ is surjective.  Then 
\begin{enumerate}
\item[(i)] $\Gamma$ is uniformly separated, and 
\item[(ii)] $D^{b+}_{\vp} (\Gamma)  < 1$ and $D^{*+}_{\vp} (\Gamma)  \le 1$. 
\end{enumerate}
\end{d-thm}

To prove Theorem \ref{pdisk-nec}, we split $\Gamma$ as a disjoint union of two sequences $\Gamma ^*$ and $\Gamma ^b$, defined by 
\[
\Gamma ^* = \left \{ \gamma \in \Gamma \ ;\ |\gamma|< \tfrac{1}{2} \right \} \quad \text{and} \quad \Gamma ^b = \left \{ \gamma \in \Gamma \ ;\ |\gamma|\ge \tfrac{1}{2} \right \}.
\]
It is clear that if $\Gamma$ is an interpolation sequence then so are $\Gamma ^*$ and $\Gamma ^b$.  The sequence $\Gamma ^*$ behaves a lot like an interpolation sequence for $(\C^*, \omega _c, \vp +2\log \log |z|^{-2})$, while the sequence $\Gamma ^b$ behaves a lot like an interpolation sequence for $(\D, \omega _P, \vp)$.  This will be our guiding principle as we proceed.

\subsubsection{Interpolation constant}\label{pdisk-interp-const-par}

As in Paragraph \ref{disk-interp-const-par}, when $\sr _{\Gamma} :\sh ^2 (\D^*, e^{-\vp}\omega _P) \to \ell ^2 (\Gamma, e^{-\vp})$ is surjective, the minimal extension operator $\se _{\Gamma} : \ell ^2 (\Gamma, e^{-\vp}) \to {\rm Kernel}(\sr _{\Gamma})^{\perp} \subset \sh ^2 (\D^*, e^{-\vp}\omega _P)$ is continuous, and has minimal norm among all bounded extension operators.  The norm 
\[
\sa _{\Gamma} := ||\se _{\Gamma}||
\]
is again called the interpolation constant of $\Gamma$.

\subsubsection{Uniform separation}

Let $\Gamma$ be an interpolation sequence. We aim to show that $\Gamma$ is uniformly separated in the sense of Definition \ref{unif-sep-density-defn}(ii).

Fix $\gamma \in \Gamma$.  As usual, we begin by choosing $F \in \sh ^2 (\D^*, e^{-\vp}\omega _P)$ such that 
\[
F(\mu) = e^{\vp(\gamma)/2}\delta _{\gamma \mu}, \ \ \mu \in \Gamma, \qquad \text{and} \quad ||F|| \le \sa _{\Gamma}.
\]
Now the proof breaks up into two cases, depending on whether $\gamma \in \Gamma ^b$ or $\gamma \in \Gamma ^*$.  
\begin{enumerate}
\item[(i)]In the first case, we can find a disk of center $\gamma$ and pseudohyperbolic radius $\delta$ (viewed as a subset of the disk), that lies in the set $|z|> 1/4$, with $\delta$ independent of $\gamma \in \Gamma ^b$.  Note that the metrics $\omega _P^{\D^*}$ and $\omega _P ^{\D}$ are quasi-isometric in the region $\{\ 1/4\le |z| < 1\}$.   An application of Proposition \ref{bergman}(b), as in the Paragraph \ref{disk-unif-sep-nec}, shows that for all $\mu \in \Gamma$, $|\phi _{\mu}(\gamma)| \ge c_o > 0$.

\item[(ii)]In the second case, where $\gamma \in \Gamma ^*$, we will imitate the proof of uniform separation for $\C^*$ carried out in \cite{v-rs1}.  To this end, we again let $\psi(z) = \vp(z) + 2 \log \log \frac{1}{|z|^2}$.  Recall that our covering map 
\[
P :\h \ni \zeta \mapsto e^{\ii \zeta}\in \D^*
\]
is the restriction of the standard covering map $P : \C \to \C^*$, and also satisfies 
\[
P^*\omega _c = \omega _o.
\]
We fix $\delta > 0$ sufficiently small that for each $\mu \in \Gamma ^*$ the disk $D^c_{\delta}(\mu):= \{ z \in \C^*\ ;\ d_c(\mu , z) < \delta\}$ 
\begin{enumerate}
\item[(a)] lies in $\D^*$, and 
\item[(b)] is the biholomorphic image under $P$ of a Euclidean disk $D^o_{\delta}(z) \subset \h$.  
\end{enumerate}Using Proposition 1.5(b) in \cite{v-rs1} (which is the Euclidean analogue of Proposition \ref{bergman}(b) above) we deduce that, in the universal cover $\h$ of $\D^*$, the preimage of any point of $\Gamma ^*-\{\gamma\}$ under $P$ is a uniform positive distance away from $\gamma$.  It follows that $d_c(\gamma, \mu) \ge c_o > 0$ for some positive constant $c_o$ independent of $\gamma$.  
\end{enumerate}
Thus $\Gamma$ is uniformly separated.
\qed

\subsubsection{Density estimates}

Now that we have established uniform separation, we move on to estimate the density of interpolation sequences.

In the rest of this paragraph, we fix a weight function $\vp \in \sC ^2 (\D^*)$ satisfying the conditions ${\rm (\star)}$ and $(B)$ of Theorem \ref{pdisk-char}, and let $\Gamma \subset \D^*$ be a closed discrete subset.  We also set 
\[
\psi = \vp + 2 \log \log \frac{1}{|\cdot|^2}.
\]
We begin with the obvious observation that if the restriction map 
\[
\sr _{\Gamma} :\sh ^2 (\D^*,e^{-\vp} \omega _P) \to \ell ^2 (\Gamma , e^{-\vp})
\]
is surjective then the restriction maps 
\[
\sr^{*} _{\Gamma^b} :\sh ^2 (\D^*,e^{-\vp} \omega _P) \to \ell ^2 (\Gamma^b , e^{-\vp}) \quad \text{and}\quad \sr^{*} _{\Gamma^*} :\sh ^2 (\D^*,e^{-\psi} \omega _c) \to \ell ^2_c (\Gamma^*, e^{-\psi})
\]
are surjective.  While we could try to adapt our arguments in the disk and the cylinder to the present setting, we prefer to use the density estimates of those cases directly.  To do so, we establish the surjectivity of the restriction map on closely related spaces.

\begin{d-thm}\label{loc-interp-glob-interp-disk}
Let $\Gamma \subset \D^*$ be uniformly separated and supported near the border.  Choose a radial function $\chi \in \sC ^{\infty}(\D)$ with $0 \le \chi \le 1$, $\chi(z) = 0$ for $|z| < c/2$ and $\chi(z) = 1$ for $|z| \ge c$, where $c\in (0,1)$ is such that $\gamma \in \Gamma \Rightarrow |\gamma| \ge c$.  Let 
\[
\hat \vp (z) = \chi(z) \vp (z) + C |z|^2,
\]
with $C$ chosen so large that $\Delta \hat \vp \ge (m+2)\omega _P$ for some positive number $m$. If the restriction map
\[
\sr_{\Gamma} :\sh ^2 (\D^*,e^{-\vp} \omega _P) \to \ell ^2 (\Gamma , e^{-\vp})
\]
is surjective, then the restriction map 
\[
\sr _{\Gamma} :\sh ^2 (\D,e^{-\hat \vp} \omega _P) \to \ell ^2 (\Gamma, e^{-\vp})
\]
is surjective.  
\end{d-thm}

\begin{proof}
Let $f \in \ell ^2 (\Gamma , e^{-\vp})$ with $||f||=1$.  By hypothesis, there exists $F \in \sh ^2 (\D^*, e^{-\vp}\omega _P)$ such that $F|_{\Gamma}= f$, and in fact, we can take $||F|| \le \sa _{\Gamma}$.  

Fix a decreasing function $h : (-\infty,\infty) \to [0,1]$ such that $h(x) = 1$ for $x < 0$, $h(x)= 0$ for $x > 2+2R$ and $|h'(x)| \le 1/2R$. Let 
\[
\xi (z) = h \left ( \log  \log \frac{1}{|z|^2} - \log \log \frac{1}{c^2}\right ).
\]
Then 
\[
\xi |_{\Gamma} \equiv 1\quad \text{and} \quad \left | \dbar \xi \right |^2_{\omega _P} = 4 \left ( h'\left (\log  \log \frac{1}{|z|^2} - \log \log \frac{1}{c^2}\right ) \right )^2 \le R^{-2}.
\]
Now, the function $\xi F$, although not globally holomorphic, interpolates $f$, and is holomorphic on the set $|z| \ge c$, but only smooth on $\D$.  Observe that  $\alpha := \dbar \xi F$ satisfies 
\[
\int _{\D} |\alpha |^2_{\omega _P} e^{-\hat \vp} \omega _P \le R^{-2} \int _{\D^*} |F|^2 e^{-\vp}\omega _P.
\]
Since $\Delta \hat \vp \ge (m+2) \omega _P$, Theorem \ref{disk-odf-thm} gives us a function $u \in L^2 (\D, e^{-\hat \vp} \omega _P)$ such that 
\[
\dbar u = \alpha\quad and \quad \int _{\D} |u|^2e^{-\hat \vp}\omega _P \le \frac{\sa _{\Gamma}^2}{R^2m}.
\]
Since $u$ is holomorphic on $|z| \ge c$, it must be small on $\Gamma$ when $R$ is large.  If we fix $R$ sufficiently large, then we obtain a function $F_1 := \xi F - u \in \sh ^2(\D , e^{-\hat \vp} \omega _P)$  such that 
\[
\sum _{\gamma \in \Gamma} |F_1(\gamma) - f(\gamma)|^2e^{-\vp (\gamma)}  \le \frac{1}{2}||f||^2 = \frac{1}{2}.
\]
(Here one uses Corollary \ref{Bergman-sums-disk}(a) and uniform separation.)

Let $f_0=f$ and $f_1 := F_1|_{\Gamma} - f$.  We continue the procedure inductively.  Assuming we have found a function $F_j \in \sh ^2(\D, e^{-\hat \vp}\omega _P)$ such that 
\[
||F_j||^2 \le \frac{C}{2^{j-1}} \quad \text{and} \quad \sum _{\gamma \in \Gamma} |F_j(\gamma) - f_{j-1}(\gamma)|^2e^{-\vp (\gamma)}  \le  \frac{1}{2^j},
\]
let $f_j := f_{j-1} - F_j|_{\Gamma}$.  Repeating the above procedure, we find $F_{j+1}$ such that 
\[
||F_{j+1}||^2 \le \frac{C}{2^j} \quad \text{and} \quad \sum _{\gamma \in \Gamma} |F_{j+1}(\gamma) - f_{j}(\gamma)|^2e^{-\vp (\gamma)}  \le  \frac{1}{2^{j+1}}.
\]
Letting 
\[
F := \sum_j F_j,
\]
we see that $F$ converges in $L^2(\D, e^{-\hat \vp}\omega_P)$, and thus locally uniformly, hence in $\sh ^2(\D, e^{-\hat \vp}\omega_P)$.  Moreover,  
\[
F|_{\Gamma}= f.
\]
The proof is finished.
\end{proof}

\begin{d-thm}\label{loc-interp-glob-interp-cyl}
Let $\Gamma \subset \D^*$ be uniformly separated and supported near the puncture.  Choose a radial function $\chi \in \sC ^{\infty}(\C)$ with $0 \le \chi \le 1$, $\chi(z) = 1$ for $|z| \le c$ and $\chi(z) = 0$ for $|z| \ge (1+c)/2$, where $c\in (0,1)$ is such that $\gamma \in \Gamma \Rightarrow |\gamma| < c$.  Let 
\[
\hat \psi(z)  = \chi(z) \psi (z) + C (1-\chi (z))(\log \tfrac{1}{|z|^2})^2,
\]
with $C>0$ chosen so large that $\Delta \hat \psi  \ge m \omega _c$ in $\C^*$.  If the restriction map 
\[
\sr_{\Gamma} :\sh ^2 (\D^*,e^{-\vp} \omega _P) \to \ell ^2 (\Gamma , e^{-\vp})
\]
is surjective, then the restriction map 
\[
\sr _{\Gamma} :\sh ^2 (\C^*,e^{-\hat \psi} \omega _c) \to \ell ^2 (\Gamma, e^{-\psi})
\]
is surjective.
\end{d-thm}

\begin{proof}
Let $f \in \ell ^2 (\Gamma , e^{-\vp})$ with $||f||=1$.  As in the proof of Theorem \ref{loc-interp-glob-interp-disk}, there exists $F \in \sh ^2 (\D^*, e^{-\vp}\omega _P)$ such that $F|_{\Gamma}= f$ and $||F|| \le \sa _{\Gamma}$.  

Fix a decreasing function $h : (-\infty,\infty) \to [0,1]$ such that $h(x) = 1$ for $x < 0$, $h(x)= 0$ for $x > 2+2R$ and $|h'(x)| \le 1/2R$. Let 
\[
\xi (z) = h \left ( \log \frac{1+|z|}{1-|z|} - \log \frac{1+c}{1-c}\right ).
\]
Then 
\[
\xi |_{\Gamma} \equiv 1\quad \text{and} \quad \left | \dbar \xi \right |^2_{\omega _P} = 4 \left ( h'\left ( \log \frac{1+|z|}{1-|z|} - \log \frac{1+c}{1-c} \right ) \right )^2 \le R^{-2}.
\]
The function $\xi F$, although not globally holomorphic, interpolates $f$, and is holomorphic on the set $|z| \le c$.  The $(0,1)$-form  $\alpha := \dbar \xi F$ satisfies 
\[
\int _{\C^*} |\alpha |^2_{\omega _P} e^{-\hat \psi} \omega _c \le R^{-2} \int _{\D^*} |F|^2 e^{-\vp}\omega _P.
\]
Since $\Delta \hat \psi \ge m \omega _c$, H\"ormander's Theorem  gives us a function $u \in L^2 (\C^*, e^{-\hat \psi} \omega _c)$ such that 
\[
\dbar u = \alpha\quad and \quad \int _{\C^*} |u|^2e^{-\hat \psi}\omega _c \le \frac{\sa _{\Gamma}^2}{R^2m}.
\]
As in the Proof of Theorem \ref{loc-interp-glob-interp-disk}, if $R$ sufficiently large, then $F_1 := \xi F - u \in \sh ^2(\C^* , e^{-\hat \psi} \omega _c)$  satisfies 
\[
\sum _{\gamma \in \Gamma} |F_1(\gamma) - f(\gamma)|^2e^{-\vp (\gamma)}  \le \frac{1}{2}||f||^2 = \frac{1}{2}.
\]
(Here one uses uniform separation and \cite[Corollary 1.6(a)]{v-rs1},which is the Euclidean analogue of \ref{Bergman-sums-disk}(a).)  The remainder of the proof is the same as that of Theorem \ref{loc-interp-glob-interp-disk}.
\end{proof}

In view of Theorems Theorem \ref{loc-interp-glob-interp-disk} and Theorem \ref{loc-interp-glob-interp-cyl}, Theorem \ref{necess-disk}, \cite[Theorem 3.7]{v-rs1}, and Definition \ref{unif-sep-density-defn}(i), we obtain the following result.

\begin{cor}\label{density-est-cor}
If $\Gamma \subset \D^*$ is an interpolation sequence, then $D^{b+}_{\vp} (\Gamma) < 1$ and $D^{*+}_{\vp} (\Gamma) \le 1$.  
\end{cor}

\noi This completes the proof of Theorem \ref{pdisk-nec} and, therefore, Theorem \ref{pdisk-char}.
\qed

\section{Interpolation in a general finite, Poincar\'e-hyperbolic Riemann surface}\label{main-proof}

For the rest of this section, we fix a finite Riemann surface $X$, i.e., the complement of finitely many points in a compact Riemann surface with (smooth) boundary.

\subsection{The ends of a finite Riemann surface covered by the unit disk}\label{ends-section}
There is a compact subset $K \relcomp X$ that is itself a Riemann surface with boundary (and in particular, has no punctures), such that the complement $X-K$ is a disjoint union of subset of $X$, called {\it ends}, each of which is biholomorphic either to an annulus or a punctured disk.  We want to describe the Poincar\'e metric $\omega _P$ on these boundary neighborhoods in a convenient way.  Doing so amounts to choosing good coordinate charts, as we now do.  For more information, see \cite{quim-rs,diller-green,sullivan}.

We fix a universal covering $\pi :\D \to X$ and denote by $G$ the associated group of deck transformations.

Let us start with the annuli, i.e., those ends whose outer boundaries are border curves of $X$.  If $\tilde \gamma$ denotes one such border curve, then there is a unique closed geodesic $\gamma$ in the homotopy class of $\tilde \gamma$.  Let $h_{\gamma}\in G$ denote the deck transformation corresponding to $[\gamma]\in \pi _1(X)$.  The quotient space 
\[
\D / \left < h _{\gamma} \right >
\]
is biholomorphic to the annulus 
\[
\mathbf{A}_R = \{ z \in \C\ ;\ e^{-R} < |z| < e^R\}, \quad \text{where }R = \pi^2 /{\rm length}(\gamma).
\]
Notice that, in $\mathbf{A}_R$, the unit circle is a geodesic of length $\pi ^2/R= {\rm length}(\gamma)$ for the Poincar\'e metric.  There is a covering map $\pi _{\gamma}:\mathbf{A}_R \to X$, which sends the unit circle in $\mathbf{A}_R$ to the geodesic $\gamma$, and maps the set 
\[
\mathbf{A}_R^{\rm outer} := \mathbf{A}_R - \D = \{ 1 \le |z| < e^R\}
\]
isometrically onto the topological annulus in $X$ bounded by $\gamma$ and $\tilde \gamma$.  The map $\pi_{\gamma}$ can be defined as follows:  a point $z$ in the annulus corresponds to an orbit $\left < h_{\gamma}\right >(\zeta)$ of some $\zeta \in \D$, and 
\[
\pi _{\gamma} (z) = \{ \text{ orbit under $G$ of the point }\zeta\}.
\]
This map is clearly well-defined, since the entire orbit $\{h^m_{\gamma}(\zeta)\ ;\ m \in \Z\}$ is contained in the orbit $G(\zeta)$.  Moreover, it is clear that $\pi_{\gamma}$ is a local isometry.  Since $\pi _{\gamma}|_{\mathbf{A}_R^{\rm outer}}$ is a bijection, we get a very precise description of the Poincar\'e metric of $X$ near the boundary $\tilde \gamma$.

Let us now turn to punctured disk ends, whose punctures are the punctures of $X$.  There is a compact Riemann surface $\tilde X$ with smooth boundary of real codimension $1$ such that $X \subset \tilde X$ and $\tilde X-X$ consists precisely of $N$ points $p_1,...,p_N$.  Let us fix one such $p=p_i$, and a neighborhood $\tilde U$ of $p$ in $\tilde X$ that is homeomorphic to a disk.  We write $U^* = \tilde U \cap X$, and note that $U^*$ is homeomorphic to a punctured disk.  There is a unique homotopy class $[\gamma]$ of a co-oriented loop $\gamma \subset U^*$ that generates the fundamental group of $U^*$.  Let $g _{[\gamma]} \in G$ denote the corresponding deck transformation of $\pi :\D \to X$.  Then we have a covering map
\[
\sigma : \D \to \D /\left < g_{[\gamma]}\right >  \cong _{\co} \D^*,
\]
and a second covering map $\pi _{p} : \D ^* \to X$  defined by 
\[
\pi_p (z) = \pi (\sigma^{-1}(z)).
\]
As in the case of the annulus, this map is well-defined, and is a local isometry.  Therefore, some neighborhood of the origin in $\D ^*$ (say $\{ 0 < |z| < c\}$) is biholomorphic and isometric to a neighborhood of $p$ in $U^*$ (which we may take to be all of $U^*$, after shrinking the original $U^*$).  

Let $A_1,...,A_k \subset X$ be neighborhoods of the $1$-dimensional boundary components, and $P_1,...,P_N\subset X$ neighborhoods of the punctures.  We also fix biholomorphic maps
\[
\pi ^{A_i}  :A_i \to \mathbf{A}_{R_i} \quad \text{and} \quad \pi ^{P_j} :P_j \to \D^*
\]
as described above.  Finally we define the core of $X$ as 
\[
X_{\rm core} := X - \left ( \bigcup _i A_i \cup \bigcup _j P_j\right ) \relcomp X.
\]
When we want to refer to an end without referring to its outer boundary dimension, we shall write $U_i$ for such an end, instead of $A_i$ or $P_i$.

As a corollary of these local descriptions, we get the following result on the local structure of the Poincar\'e metric near the boundary of a finite Riemann surface.

\begin{prop}\label{asymp-poincs}
Let $X$ be a finite Riemann surface.  Then there is a compact set $K \relcomp X$ whose complement is a union of ends $U_i$ as just described, and subsets $V_i \subset U_i$ whose closure contains $\di X \cap \overline{U_i}$ but does not meet $\di K$, such that the Poincar\'e metric $\omega _P$ of $X$ and the Poincar\'e metric $\omega _{P,i}$ of $U_i$ satisfy 
\[
\omega _{P,i}|_{V_i} = \omega _P|_{V_i}.
\]
\end{prop}

\subsection{Ortega Cerd\`a's Theorem}\label{past-results}

As we mentioned in the introduction, Ortega Cerd\`a \cite{quim-rs} proved Theorem \ref{main} in the case where $X$ has no punctures, and at least one border curve.  Though he never explicitly says it, Ortega Cerd\`a normalizes the Poincar\'e metric to have constant curvature $-2$.

In fact, Ortega Cerd\`a's main theorem is proved for $L^{\infty}$, and he then sketches how the same methods can be used to establish the $L^p$ case.  The crucial result he needs to carry out his proof in the $L^p$ case is the following theorem, which he proves. 

\begin{d-thm}\cite[Theorem 17]{quim-rs}\label{aim-dbar-lp}
Let $X$ be a finite open Riemann surface with no punctures, and with Poincar\'e metric $\omega _P$ (of constant curvature $-4$), let $\phi$ be a function satisfying  $(2+\ve)\omega _P \le \Delta \phi \le M\omega _P$ for some positive constants $M$ and $\ve$, and let $p \in [1,\infty)$.  Then there is a constant $C > 0$ such that for any locally integrable $(0,1)$-form $\alpha$ on $X$ there exists a function $u \in L^p_{\ell oc}(X)$ such that 
\[
\dbar u = \alpha \quad \text{and} \quad \int _X |u|^p e^{-\phi} \omega _P \le C \int _X |\alpha |^p_{\omega _P} e^{-\phi} \omega _P,
\]
provided the right hand side is finite.
\end{d-thm}

At least for $L^2$, it is possible to prove a somewhat stronger statement using the technique of Donnelly-Fefferman-Ohsawa.  The result we prove is stronger in two senses.  Firstly, we don't need to assume that $X$ has no punctures, and secondly, we do not assume an upper bound on the Laplacian of the weight.  (The bounded Laplacian condition implies that the weight function is $\sC ^{1,\alpha}$ for any $\alpha < 1$, but, as the next result shows, when $p=2$ the Theorem \ref{aim-dbar-lp} holds for weights that are much more general.)

\begin{d-thm}\label{odf-gen}
Let $X$ be a finite open Riemann surface with Poincar\'e metric $\omega _P$ (again of constant curvature $-4$), and let $\xi$ be a weight function satisfying  $\Delta \xi \ge 2(1+\ve)\omega _P$ for some positive constant $\ve$.  Then there is a constant $C > 0$, depending only on $X$ and $\ve$, such that for any locally integrable $(0,1)$-form $\alpha$ on $X$ there exists a function $u \in L^1_{\ell oc}(X)$ such that 
\[
\dbar u = \alpha \quad \text{and} \quad \int _X |u|^2 e^{-\xi} \omega _P \le C \int _X |\alpha |^2_{\omega _P} e^{-\xi} \omega _P,
\]
provided the right hand side is finite.
\end{d-thm}

\begin{rmk}
As with Theorems \ref{disk-odf-thm} and \ref{punctured-disk-odf-thm}, Theorem \ref{odf-gen} follows from H\"ormander's Theorem if we assume $\Delta \xi \ge 4(1+\ve)\omega _P$, but not in general.   
\red
\end{rmk}

In the proof of Theorem \ref{odf-gen}, as well as elsewhere, we will need the following lemma.

\begin{lem}\label{bumper}
Let $X$ be a finite open Riemann surface and let $\Omega$ be a continuous $(1,1)$-form with compact support in $X$.  Then for any $\ve > 0$ there exists a constant $C>0$ and a smooth function $\tau : X \to (0,C)$ such that 
\[
\Delta \tau \ge -\Omega - \ve \omega _P.
\]
\end{lem}

\begin{proof}
Observe that $X$ is an open subset of a compact Riemann surface $Y$.  Thus there exists a smooth metric of strictly positive curvature for some holomorphic line bundle, say $L \to Y$.  By Kodaira's Embedding Theorem, if $k \in \N$ is sufficiently large then the sections of $L^{\tensor k} \to Y$, embed $Y$ in some projective space.  If we take a basis of sections $\sigma^{(k)} _0,...,\sigma^{(k)} _{N_k} \in H^0(Y, L^{\tensor k})$, we can form the metric 
\[
\psi_k := \log  \sum _{j=0} ^{N_k}  |\sigma^{(k)}_j|^2.
\]
The curvature of this metric is as large as we like on any compact subset of $X$, and in particular, for sufficiently large $k$ we have 
\[
\Delta \psi _k \ge - \Omega
\]
on the support of $\Omega$.  

We can choose the first section $\sigma^{(k)} _0$ to have no zeros in $X$.  Then we can define functions $f^{(k)}_j := \sigma^{(k)} _j / \sigma^{(k)} _0$, $0 \le j \le N_k$ for some constant $c$.  The function 
\[
\Psi _k = \log (1+|f^{(k)}_1|^2+...+|f^{(k)}_{N_k}|^2)
\]
is the local trivialization of the metric $\psi _k$ with respect to the frame $\sigma ^{(k)}_0$.  It is a smooth function on the set $Y-\{\sigma^{(k)}_0=0\}$, and in particular, on $X$.  Thus if the zeros of $\sigma ^{(k)}_0$ are not on the boundary of $X$, this function is bounded.  In particular, if $X$ has at least one border curve, then we are done (and even better: we can take $\ve = 0$).

However, if (and only if) $X$ has only punctures , it is impossible to avoid placing the zeros of $\sigma^{(k)}_0$ on the boundary of $X$.  Let us assume that $Y-X$ consists of finitely many points $p_1,..., p_m$.  We choose our section $\sigma ^{(k)}_0$ to have its only zero at $p_1$ (evidently with high multiplicity).  Let us choose coordinates $z$ on the punctured neighborhood $U_1$ of $p_1$, as defined in Subsection \ref{ends-section}, and fix $c > 0$ such that $\{|z|\le c\} \cap {\rm Support}(\Omega) = \emptyset$.  Now let $h: (-\infty , \infty) \to [0,1]$ be a smooth increasing function such that $h(x) = 1$ for $x \ge 0$,  $h(x) = 0$ for $x \le -R-1$, $h'(x) \le R^{-1}$ and $|h''(x)| \le C R^{-2}$.  Let 
\[
\kappa (z) := \chi (\log \log \tfrac{1}{|z|^2} - \log \log \tfrac{1}{c^2})
\]
for $|z| \le c$, and $\kappa (p) = 1$ for $p \in X - \{ |z| \le c\}$.  Note that, for $R$ large,  
\[
|\di \kappa (z)|^2_{\omega _P} \lesssim R^{-2},
\]
and that 
\[ 
\Delta  \kappa = - 4 \chi '(\log \log \tfrac{1}{|z|^2} - \log \log \tfrac{1}{c^2})\omega _P + \chi '' (\log \log \tfrac{1}{|z|^2} - \log \log \tfrac{1}{c^2}) \omega _P \sim - 4R^{-1} \omega _P.
\]
Consider the function smooth, compactly supported (and hence, bounded) function
\[
\tau  = \kappa \cdot \Psi _k.
\]
Then 
\[
\Delta \tau = \Delta \kappa \Psi _k + \ii (\di \kappa \wedge \dbar \Psi _k + \di \Psi _k \wedge \dbar \kappa) + \kappa \Delta \Psi _k \ge - \Omega + O(R^{-1})(-\omega _P).
\]
The result follows as soon as $R$ is sufficiently large.
\end{proof}

\begin{proof}[Proof of Theorem \ref{odf-gen}]
In each end $U_j$ with coordinates coming from the universal cover as described in the previous section, we have a function $\eta_j$, which we take to be $\log \log \frac{1}{|z|^2}$ or $\log \frac{1}{1-|z|^2}$, depending on whether $U_j$ is a punctured disk or an annulus respectively.  Fix cutoff functions $\chi _j$ that are identically $1$ in $V_j$, take values in $[0,1]$, and are supported in $U_j$.  Define the function 
\[
\eta := \sum _j \chi _j \eta _j.
\]
Since $\Delta \eta _j - (1+\nu) \ii \di \eta _j \wedge \dbar \eta _j \ge - 2 \nu \omega _P$, 
\[
\Delta \eta - \ii \di \eta \wedge \dbar \eta \ge - 2\nu \omega _P +  \Xi
\]
for some smooth $(1,1)$-form $\Xi$ with compact support in $X$.  (If fact, $\Xi$ is supported in the union of the annuli $U_j - V_j$.)  We shall now apply Theorem \ref{odf-thm} with $\psi = \xi +\eta +\tau$ where $\tau$ is as in Lemma \ref{bumper} with $\Omega = \Xi$, and $\omega = \omega_P$.  We calculate that for $\nu$ sufficiently small and $\tau$ appropriately chosen, 
\[
\Delta (\psi +{\rm R}(\omega)  +\Delta \eta - (1+\nu)\di \eta \wedge \dbar \eta \ge \Delta \xi - 2 \omega _P +\Delta \tau +\Xi - 2 \nu \omega _P \ge \ve \omega _P.
\]
Thus we may take $\Theta = \ve \omega _P$ in Theorem \ref{odf-thm}.  Since $0 \le \tau \le C$ for some constant $C$ that depends only on $X$ and $\ve$, we find that if 
\[
\int_X e^{-\xi}|\alpha|^2_{\omega _P} \omega _P  \le  \int_X e^{-\tau-\xi}|\alpha|^2_{\omega _P} \omega _P = \ve  \int_X e^{\eta-\psi}|\alpha|^2_{\Theta} e^{-\psi} \omega _P  <+\infty
\]
there exists a locally integrable function $u$ such that $\dbar u = \alpha$, and  
\begin{eqnarray*}
\int _{X} e^{-\xi} |u|^2\omega _P &\le& e^C \int_X e^{-\tau-\xi}|u|^2 \omega _P \\
&\le& \frac{\nu+1}{\nu}e^C  \int_X e^{-\tau-\xi}|\alpha|^2_{\Theta}\\
&=& \frac{\nu+1}{\nu}e^C \ve ^{-1} \int_X e^{-\tau-\xi}|\alpha|^2_{\omega _P} \omega _P\\
&\le& \frac{\nu+1}{\nu}e^C \ve ^{-1} \int_X e^{-\xi}|\alpha|^2_{\omega _P} \omega _P.
\end{eqnarray*}
This completes the proof.
\end{proof}

\begin{rmk}
The technique of the proof, in particular the use of Lemma \ref{bumper}, works perfectly well if we require $\Delta \xi \ge 2(1+\ve)\omega _P$ to hold only outside a given compact subset of $X$, and allow $\Delta \xi$ to be negative in the interior of $X$, so long as $\xi$ is quasi-subharmonic, i.e., $\Delta \xi$ is bounded below by a smooth negative $(1,1)$-form.
\red
\end{rmk}

\begin{rmk}
Note also that although Theorem \ref{odf-gen} to some extent generalizes Theorems \ref{disk-odf-thm} and \ref{punctured-disk-odf-thm}, the constants in the latter are sharper.  The reason is that the function $\eta$ constructed in the proof only satisfies $\Delta \eta \ge \ii \di \eta \wedge \dbar \eta$ in the ends of $X$, but in general it is negative in the interior.  We do not know if it is possible to find a real-valued function $\eta \in W^{1,2}_{\ell oc}(X)$ such that 
\[
\Delta \eta \ge \ii \di \eta \wedge \dbar \eta \quad \text{and} \quad \Delta \eta = 2\omega _P,
\]
except when $X=\D$ or $X=\D^*$.  Real-valued functions satisfying the inequality 
\[
C \Delta \eta \ge \ii \di \eta \wedge \dbar \eta
\]
for some constant $C>0$ were first introduced by McNeal \cite[Definition 1]{jeff-sbg} in his work on the $\dbar$-Neumann problem.  McNeal called them "functions with self-bounded complex gradient".
\red
\end{rmk}

\subsection{Uniform separation and asymptotic density}

Since all the ends of our finite Riemann surface $X$ are either bordered or punctured ends, we can import the notions of uniform separation and of density from the work we did on the punctured disk in the previous section.

In each end $U_j$, we have an open set $V_j$ which is biholomorphic either to an annulus or a punctured disk under the map $\pi ^{U_j}$.  We shall think of $\pi ^{U_j}(V_j)$ as a subset of $\D^*$, which is either supported near the border or near the puncture.  

For each $j$, the weight function $\vp _j := ((\pi ^{U_j}|_{V_j})_{*}\vp$ satisfies the hypotheses of Interpolation Theorem \ref{pdisk-char} on the image of $\pi ^{U_j}$.  We define the sequences $\Gamma _j := \pi ^{U_j}(\Gamma \cap V_j)$.  Based on Definition \ref{unif-sep-density-defn},  we are now ready to define uniform separation and asymptotic density of $\Gamma$.

\begin{defn}
Let $\Gamma \subset X$ be a closed discrete subset.
\begin{enumerate}
\item[(i)]  We say $\Gamma$ is uniformly separated if each $\Gamma _j \subset \D^*$ is uniformly separated according to Definition \ref{unif-sep-density-defn}.
\item[(ii)] The number  
\[
D^+_{\vp} (\Gamma) := \max _j \dot D^+_{\vp_j}(\Gamma _j)
\]
is called the {\it asymptotic (upper) density} of $\Gamma$ with respect to $\vp$.
\red
\end{enumerate}
\end{defn}

\subsection{Necessity}

Conveniently, necessity of the conditions of Theorem \ref{main} follow rather easily from the special case of the punctured disk.  We therefore begin with necessity.

\subsubsection{Uniform separation of interpolation sequences}

\begin{prop}
If $\Gamma$ is an interpolation sequence then $\Gamma$ is uniformly separated.
\end{prop}

\begin{proof}
Clearly, for each $j$, $\Gamma \cap {U_j}$ is then an interpolation sequence in $\D^*$ that is supported either near the border or near the puncture.  It follows that each $\Gamma \cap {U_j}$ is uniformly separated.  Since $\Gamma$ is a closed discrete subset, $\Gamma - \bigcup _j \Gamma \cap U_j$ is finite.  Therefore $\Gamma$ is uniformly separated.
\end{proof}

\subsubsection{Density bound for interpolation sequences}

\begin{prop}
If $\Gamma$ is an interpolation sequence then $D^+_{\vp}(\Gamma) \le 1$.
\end{prop}

\begin{proof}
Again, for each $j$, $\Gamma \cap {U_j}$ is an interpolation sequence in $\D^*$ that is supported either near the border or near the puncture.  Thus $\dot D^+_{\vp_j} (\Gamma \cap U_j) \le 1$ for all $j$.  That is to say, $D^+_{\vp}(\Gamma) \le 1$.  Of course, if there are only border-type boundary components, then $D^+_{\vp}(\Gamma) < 1$.
\end{proof}

\subsection{Sufficiency}

We shall follow the approach used to prove Theorem \ref{pdisk-ssuff}, which is the case of the punctured disk.

\subsubsection{Raw densities}

Our definitions of the upper density placed a condition on the Laplacian of some average of the weight.  If we use $\vp$ without averaging, the definition can still make sense.  In \cite{v-rs1} we called the resulting density the {\it raw density}.  The precise definition is 
\[
\check D ^{b+}_{\vp} (\Gamma) := \inf \left \{ \frac{1}{\alpha} > 0 \ ;\ \Delta \vp - 2\omega _P \ge \alpha \Upsilon ^{b, \Gamma}_r \right \}
\]
if $\Gamma$ is supported near a border curve, and 
\[
\check D ^{*+}_{\vp} (\Gamma) := \inf \left \{ \frac{1}{\alpha} > 0 \ ;\ \Delta \vp - 4\omega _P \ge \alpha \Upsilon ^{*, \Gamma}_r \right \}
\]
if $\Gamma$ is supported near a puncture.  (See Paragraph \ref{border-sing-par} for the definition of $\Upsilon ^{b, \Gamma}_r$, and Paragraph \ref{p-sing-par} for the definition of $\Upsilon ^{*, \Gamma}_r$.)  In the general case, the raw density 
\[
\check D^+_{\vp} (\Gamma)
\]
is the maximum of the (finitely many) raw densities of the sequences $\Gamma \cap U_j$

\subsubsection{Singularities along $\Gamma$}

Let ${\tilde T} \in \co (X)$ be any holomorphic function such that 
\[
{\rm Ord}({\tilde T}) = \Gamma.
\]
Now, in an annular neighborhood $A_j$ of a border curve, using our isometric coordinates, we can define a function $\lambda _j$ on $U_j$, which agrees, on $V_j$ (the outer part of the annulus) with the function $\lambda ^{\tilde T}_{r_j}$ defined in Paragraph \ref{border-sing-par}.  

In a punctured neighborhood $P_j$, we have another such function, $\lambda _j$, which agrees with the function $\tilde \lambda ^{\tilde T}_{r_j}$ in Paragraph \ref{p-sing-par}.

We then define a function $\lambda$ by cutting off the $\lambda _j$ and summing:
\[
\lambda  := \sum _{i=1} ^{n+m} \chi _j \lambda _j.
\]
Here $\chi _j$ is smooth, takes values in $[0,1]$, is supported in $U_j$, and is identically $1$ on $V_j$.

Let 
\[
L := X - \bigcup _j V_j.
\]
Then $L$ is compact, and therefore there is a positive constant $M$ such that 
\[
\log |{\tilde T}|^2 - \lambda  \le M \quad \text{on }L.
\]
On the other hand, the sub-mean value property for subharmonic functions implies that 
\[
\log |{\tilde T}|^2 - \lambda  \le 0 \quad \text{on each }V_j.
\]
Therefore 
\[
\log |{\tilde T}|^2 - \lambda \le M \quad \text{on }X.
\]
Letting $T := e^{-M} \tilde T$ (but keeping $\tilde T$ in the definition of $\lambda$), we have found functions $T$ and $\lambda$ such that 
\[
{\rm Ord}(T)= \Gamma \quad \text{and} \quad \sigma := |T|^2e^{-\lambda} \le 1.
\]

\subsubsection{Strong sufficiency}

We shall now prove the following theorem.  

\begin{d-thm}[Stong sufficiency: general case]\label{semi-strong-suff-general}
Let $X$ be a finite Riemann surface covered by the disk, and let $\vp \in L^1_{\ell oc}(X)$ be a weight satisfying the curvature hypotheses 
\begin{enumerate}
\item[(o)] $\Delta \vp \ge - \Theta$ for some smooth, nonnegative $(1,1)$-form $\Theta$, 
\item[(i)] $\Delta \vp -2\omega_P \ge m \omega_P$ in some annular neighborhood of each border curve, and 
\item[(ii)] $\Delta \vp - 4 \omega _P \ge m \omega _c$ in some punctured neighborhood of each puncture of $X$,
\end{enumerate}
for some constant $m>0$.  Assume $\Gamma \subset X$ is uniformly separated, and that 
\[
\check D^+_{\vp} (\Gamma) < 1.
\]
Then the restriction map $\sr _{\Gamma} :\sh ^2 (X,e^{-\vp} \omega )  \to \ell ^2 (\Gamma , e^{-\vp})$ is surjective.
\end{d-thm}

\begin{proof}
Let $f \in \ell ^2(\Gamma, e^{-\vp})$ be the datum to be extended.

As in the proof of the special case of the punctured disk, the density condition implies that there are holomorphic functions $F_j \in \co (U_j)$ such that 
\[
F_j(\gamma) = f(\gamma) , \quad \gamma \in V_j 
\]
and 
\[
\int _{U_j} |F_j|^2 e^{-\vp_j} \omega _P < +\infty.
\]

Next, let $\tilde L \relcomp X$ be a compact set with $L \relcomp {\rm interior}(\tilde L)$ and $\Gamma \cap (\tilde L - L) = \emptyset$.  Using the $L^2$ extension theorem locally, and then H\"ormander's Theorem, it is straightforward to construct a holomorphic function $F_o\in \co(X)$ such that 
\[
F_o|_{\Gamma \cap L} = f|_{\Gamma \cap L} \quad \text{and} \quad \int_{\tilde L} |F_o|^2e^{-\vp}\omega _P < +\infty.
\]

Now we wish to glue together all the extensions $F_o, F_1,...,F_N$ to produce an extension $F$ of $f$.  We can choose cut-off functions $\chi _o, \chi _1,...,\chi _N$ with $\chi _o |_L \equiv 1$ and $\chi _j|_{V_j}\equiv 1$, such that the function 
\[
\tilde F := \chi _o F_o + \chi _1 F_1 + ... \chi _N F_N
\]
is
\begin{enumerate}
\item[(a)] smooth, 
\item[(b)] holomorphic everywhere except possibly along collars connecting the ends $V_j$ to $L$ (which therefore do not meet $\Gamma$), and
\item[(c)] satisfies $\tilde F|_{\Gamma} = f$.
\end{enumerate}
The smooth form $\alpha := \dbar \tilde F$ is compactly supported, and satisfies 
\[
\int _{X} |\alpha|^2 _{\omega _P} e^{-\vp} \omega _P < +\infty.
\]
We wish to find $u \in L^2 (X, e^{-\vp} \omega)$ such that $\dbar u = \alpha$ and $u |_{\Gamma}= 0$.  To do so, we shall use the singular weight 
\[
\xi := \vp + \log |T|^2 - \lambda + \tau,
\]
where $\tau$ is chosen as in Lemma \ref{bumper} with respect to a form $\Omega$, with $\Omega$ to be chosen momentarily.  We know that $\xi \le \vp + \max \tau$.  Next we calculate, as in the proof of Theorem \ref{pdisk-ssuff}, that for sufficiently large $r_j$ (used in the definition of $\lambda$ above), 
\[
\Delta \xi - 2\omega _P \ge   \Delta \vp -2\omega _P -\Delta \lambda  +\Delta \tau \ge m \omega _P +\Omega +\Delta \tau,
\]
where $\Omega$ is a continuous form with compact support in the interior of $\tilde L$.  Theorem \ref{odf-gen} therefore implies the existence of a function $u \in L^1_{\ell oc}(X)$ such that 
\[
\dbar u = \alpha \quad \text{and} \quad \int _X |u|^2e^{-\vp} \omega _P \le \int _X |u|^2 e^{-\psi} \omega _P < +\infty.
\]
Again $\dbar u = \alpha$ implies that $u$ is smooth, and the finiteness of the second integral means that $u$ must vanish on $\Gamma$.  Therefore 
\[
F := \tilde F - u
\]
solves the interpolation problem for $f$, and our proof is complete.
\end{proof}

\subsubsection{Conclusion of the proof of Theorem \ref{main}}

To obtain the sufficiency part of Theorem \ref{main}, we need to replace $\vp$ by some sort of average $\vp _r$ of $\vp$ such that 
\begin{enumerate}
\item[(i)] $\vp _r$ still satisfies the curvature conditions $(B)$ and $(\star)$ of Theorem \ref{main}, and 
\item[(ii)] $\sh ^2(X, e^{-\vp_r}\omega) \cong \sh ^2(X,e^{-\vp}\omega)$ and $\ell ^2(\Gamma, e^{-\vp_r}) \cong \ell ^2(\Gamma,e^{-\vp})$ as topological vector spaces, i.e., the isomorphisms are bounded linear maps.
\end{enumerate}

We already know how to do this in the ends, since we have done so in the punctured disk.  In the interior it doesn't matter how we do it, since densities are determined at the ends.  For the sake of deciding on one method, we can cover our compact set $\tilde L$ by a finite number of open coordinate charts biholomorphic to disks, and simply replace $\vp$ by its average over a disk of some fixed radius.

After averaging $\vp$ in this way, we multiply the $\vp _{i,r}$ of the end by the cutoff functions $\chi _i$, and multiply the interior averages by any smooth cutoff functions that give a partition of unity on $\tilde L$.  (Again, what we do in the interior is not so important.)  If we now sum up all of the cut off averages to form $\tilde \vp_r:= \sum _i \vp _{i,r}$, then clearly 
\[
D^+_{\vp} (\Gamma) = \check D ^+_{\tilde \vp_r} (\Gamma).
\]
The proof of Theorem \ref{main} is complete.
\qed

\section{Remark on Shapiro-Shields Interpolation}

In \cite{v-rs1} we mentioned that there is another, more classical, theory on interpolation, which is well-defined in the category of Hilbert spaces of holomorphic functions over a Riemann surface (or more general complex manifold).  We refer the reader to \cite{v-rs1} for the definition of Shapiro-Shields Interpolation.  We showed there that in the asymptotically flat case, our notion of interpolation coincides with Shapiro-Shields interpolation.  

As it turns out, the same is true for the more general interpolation of the present paper, i.e., it is equivalent to Shapiro-Shields Interpolation.  Following the ideas of \cite{v-rs1}, it suffices to establish the following proposition.

\begin{prop}\label{ssi}
Let $(X,\omega_P)$ be a finite Hyperbolic Riemann surface, and suppose the weight function $\vp$ satisfies the curvature hypotheses $(\star)$ and $(B)$ of Theorem \ref{main}.  Then there is a constant $C$ such that 
\[
C^{-1} \le K(z,z)e^{-\vp(z)} A_{\omega_P}(z)\le C,
\]
where $K$ is the Kernel of the Bergman projection $P: L^2(X, e^{-\vp}\omega) \to \sh ^2 (X, e^{-\vp} \omega)$.
\end{prop} 

As we mentioned in \cite{v-rs1}, the equivalence of our notion of interpolation with the Shapiro-Shields notion, though interesting, is not strictly necessary for the present article.  For this reason, and since the proof of Proposition \ref{ssi} is almost directly analogous to its asymptotically flat analog, we will content ourselves here with just a sketch, and leave details to the interested reader.

As in \cite{v-rs1}, it suffices to establish the estimates at the ends.  In a bordered end, the upper bound is obtained as in the flat case, but using Ohsawa's Theorem in place of H\"ormander's Theorem.  The use of Ohsawa's Theorem is possible because of Condition $(B)$.  The lower bound is softer, and obtained from the work in Paragraph \ref{bdd-lap} in a manner analogous to the flat case.

In a punctured end, one uses the fact that $A_{\omega_P}(z) \sim (\log |z|^2)^{-2}$.  This transforms the punctured end to a cylindrical end, and the results of \cite{v-rs1} apply directly.

\end{document}